\documentclass[12 pt]{article}
\usepackage{amsmath}
\usepackage{amsfonts}
\usepackage{amssymb}
\usepackage{amsthm}
\usepackage{thmtools}
\usepackage{hyperref}
\usepackage{mathtools}
\usepackage{enumitem}
\usepackage{subcaption}
\usepackage[style=numeric,backend=bibtex]{biblatex}

\newtheorem{thm}{Theorem}[section]
\newtheorem{lem}[thm]{Lemma}
\newtheorem{ques}[thm]{Question}
\newtheorem{cor}[thm]{Corollary}
\newtheorem{rmk}[thm]{Remark}
\newtheorem{claim}[thm]{Claim}
\allowdisplaybreaks

\DeclareMathOperator{\V}{Vert}
\DeclareMathOperator{\rk}{rk}

\title{Heegaard Floer homology and chirally cosmetic surgeries}
\author{Konstantinos Varvarezos}
\begin{document}
\maketitle

\begin{abstract}
A pair of surgeries on a knot is chirally cosmetic if they result in homeomorphic manifolds with opposite orientations.  We find new obstructions to the existence of such surgeries coming from Heegaard Floer homology; in particular, we make use of  immersed curve formulations of knot Floer homology and the corresponding surgery formula.  As an application, we completely classify chirallly cosmetic surgeries on odd alternating pretzel knots, and we rule out such surgeries for a large class of Whitehead doubles. Furthermore, we rule out cosmetic surgeries for L-space knots along slopes with opposite signs.
\end{abstract}


\section{Introduction}

Given a knot $K$ in $S^3$ and an $r\in \mathbf{Q}\cup \{\infty\},$ we denote the \textit{Dehn surgery} on $K$ with slope $r$ by $S^3_r(K).$  Surgeries on $K$ along distinct slopes $r$ and $r'$ are called \textit{cosmetic} if $S^3_r(K)$ and $S^3_{r'}(K)$ are homeomorphic manifolds.  Furthermore, a pair of such surgeries is said to be \textit{purely cosmetic} if $S^3_r(K)$ and $S^3_{r'}(K)$ are homeomorphic as \textit{oriented} manifolds.  We use the symbol $\cong$ to denote ``orientation-preserving homeomorphic." If, on the other hand, $S^3_r(K) \cong -S^3_{r'}(K),$ we say this pair of surgeries is \textit{chirally cosmetic}; here $-M$ denotes the manifold $M$ with the opposite orientation.

No purely cosmetic surgeries are have been found on nontrivial knots in $S^3$; compare Problem 1.81(A) in \cite{KirbyList}.  Indeed, many powerful obstructions to purely cosmetic surgeries are known. Heegaard Floer techniques have been particularly effective for this; see \cite{NW}, and more recently, immersed curve versions of Heegaard Floer invariants were employed by Hanselmant to give even stronger obstructionss \cite{Hans}.  On the other hand, there are examples of chirally cosmetic surgeries.  For instance, $S^3_r(K) \cong -S^3_{-r}(K)$ whenever $K$ is an amphicheiral knot.  Also, $(2,n)$-torus knots are known to admit chirally cosmetic surgeries, precisely along the pairs of slopes $\frac{2n^2(2m+1)}{n(2m+1) \pm 1}$ for each positive integer $m$; see \cite{IIS,Mat}.  On the other hand, several families of knots have been shown not to admit any chirally cosmetic surgeries: genus-1 alternating knots \cite{IIS}, certain cables of nontrivial knots \cite{Ito2}, and certain pretzel knots \cite{CCP}. In fact, the following question remains open:

\begin{ques}\label{ques:gen}
Suppose a knot $K$ is not amphichiral and is not a $(2,n)$-torus knot.  Does $K$ admit any chirally cosmetic surgeries?
\end{ques}

Notice that in the case of the $(2,n)$-torus knots, each pair of  slopes which gives rise to chirally cosmetic surgeries has slopes with the same sign. Indeed, previously known obstructions coming from Heegaard Floer homology imply that knots which admit chirally cosmetic surgeries along slopes with the same sign are rare:

\begin{thm}[Theorem 9.8 of \cite{OSzRat}]\label{thm:rk}
Let $K$ be a knot and suppose $S^3_r(K)\cong \pm S^3_{r'}(K)$ for two distinct slopes $r,r'.$  Then either $r$ and $r'$ have opposite signs or $S^3_r(K)$ is an L-space (hence $K$ is an L-space knot).
\end{thm}
Being an L-space knot is a very restrictive condition; it implies, for instance, fiberedness \cite{Ni} as well as an Alexander polynomial with coefficients in $\lbrace-1,0,-1\rbrace$ (see \cite{OSzRat} for a more precise constraint). One may hope to approach Question \ref{ques:gen}, therefore, by investigating the following:

\begin{ques}\label{ques:opp}
Suppose a knot $K$ is not amphichiral. Does $K$ admit any chirally cosmetic surgeries along slopes with opposite signs?
\end{ques}
A negative answer to this question would reduce Question \ref{ques:gen} to classifying cosmetic surgeries of L-space knots.  To this end, our new obstruction is a step towards a better understanding of Question \ref{ques:opp}. In particular,	 we obtain a restriction on the possible cosmetic surgeries along slopes of opposite signs for knots with Heegaard Floer-theoretic invariant tau equal to their genus.

\begin{restatable}{thm}{main}\label{thm:main}
Suppose a knot $K$ satisfies $\tau(K)=g(K)>0.$  If $S^3(\frac{p}{q})\cong-S^3(\frac{p}{q'})$ where $p,q,q \in \mathbf{Z}$ with $p,q>0$ and $q'<0,$ then $\frac{p}{q}\leq 2g(K)-1$ and $2p=(\V(K)+2g(K)-1)(q+q').$
\end{restatable}
Of course, non-amphichiral need not satisfy $\tau=0,$ much less have tau equal to genus.  As amphichiral knots satisfy $\tau=0,$ the hypotheses of this result apply to knots which are ``very far" from being amphichiral, in a certain sense.  Furthermore, even though the restriction on surgery slopes is \textit{a priori} an infinite set of pairs, it is possible to combine it with a similar kind of obstruction arising from finite type invariants to exclude all slopes in many situations; see section \ref{sec:combo} below for more details.

Our first application  of this obstruction is to the family of \textit{alternating odd pretzel knots}.  Up to mirror image, these are pretzel knots of the form ${P(-2k_1-1,2k_2-1,\dots,-2k_{2g+1}-1)},$ where the $k_i$ are all nonnegative.  See Figure \ref{fig:Pknot} for an illustration.  It is a fact that $g$ is the Seifert genus of the knot.  For convenience, we will often use the shorthand $K(k_1,k_2,\dots,k_{2g+1})$ to denote these knots.

\begin{figure}
\centering
\includegraphics[width=.75\textwidth]{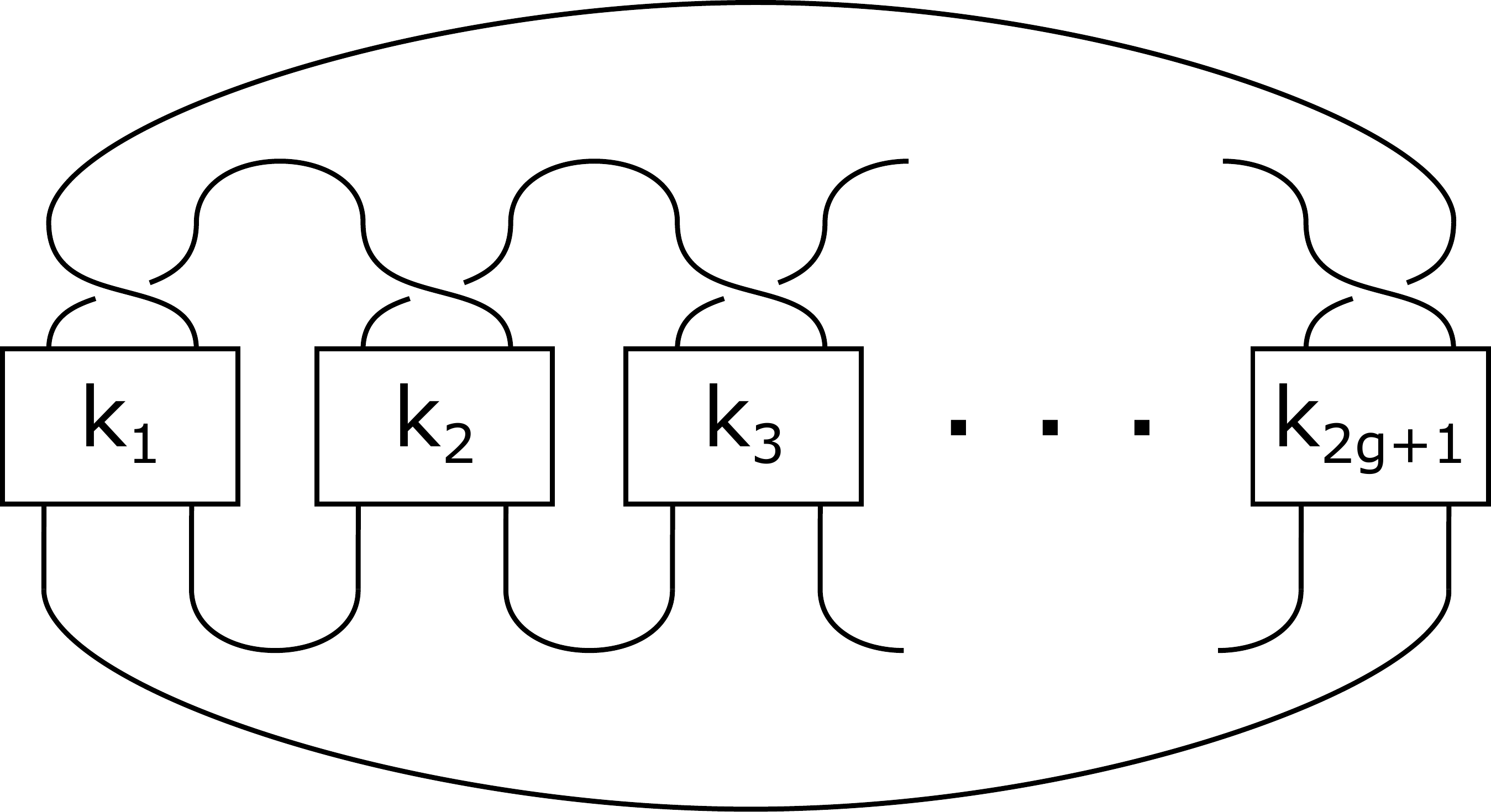}
\caption{The pretzel knot $P(-2k_1-1,-2k_2-1,-2k_3-1,\dots,-2k_{2g+1}-1).$  Here the boxes represent $k_i$ full left-handed twists of the strands passing through.}
\label{fig:Pknot}
\end{figure}

These knots are already known not to admit any purely cosmetic surgeries. % This can be seen, for instance, from the fact that the signatures and $v_3$ of these knots are nonzero(see section \ref{sec:comp} below).  By \cite{Hans}, alternating knots that admit purely cosmetic surgeries must have zero signature.  Alternatively, by \cite{IW}, knots admitting purely cosmetic surgeries must have $v_3=0.$

\begin{restatable}{thm}{pretz}\label{thm:pretz}
Let $K=K(k_1,\dots,k_{2g+1})$ where $k_i\geq 0$ for all $i.$  If at least one $k_i>0,$ then $K$ does not admit any chirally cosmetic surgeries.
\end{restatable}
\begin{rmk}
The knots $K(0,0,\dots,0)$ are the $(2,2g+1)$-torus knots, which are already known to admit (infinitely many pairs of) chirally cosmetic surgeries; see Corollary A.2 of \cite{IIS} for a full classification of the chirally cosmetic surgeries for these knots.
\end{rmk}
%Ichihara, Ito, and Saito have already shown the analogue of this for the genus 1 (three-stranded) case.  Indeed, they classified chirally cosmetic surgeries for all alternating genus 1 knots \cite{IIS}.

\begin{figure}
\centering
\includegraphics[width=.75\textwidth]{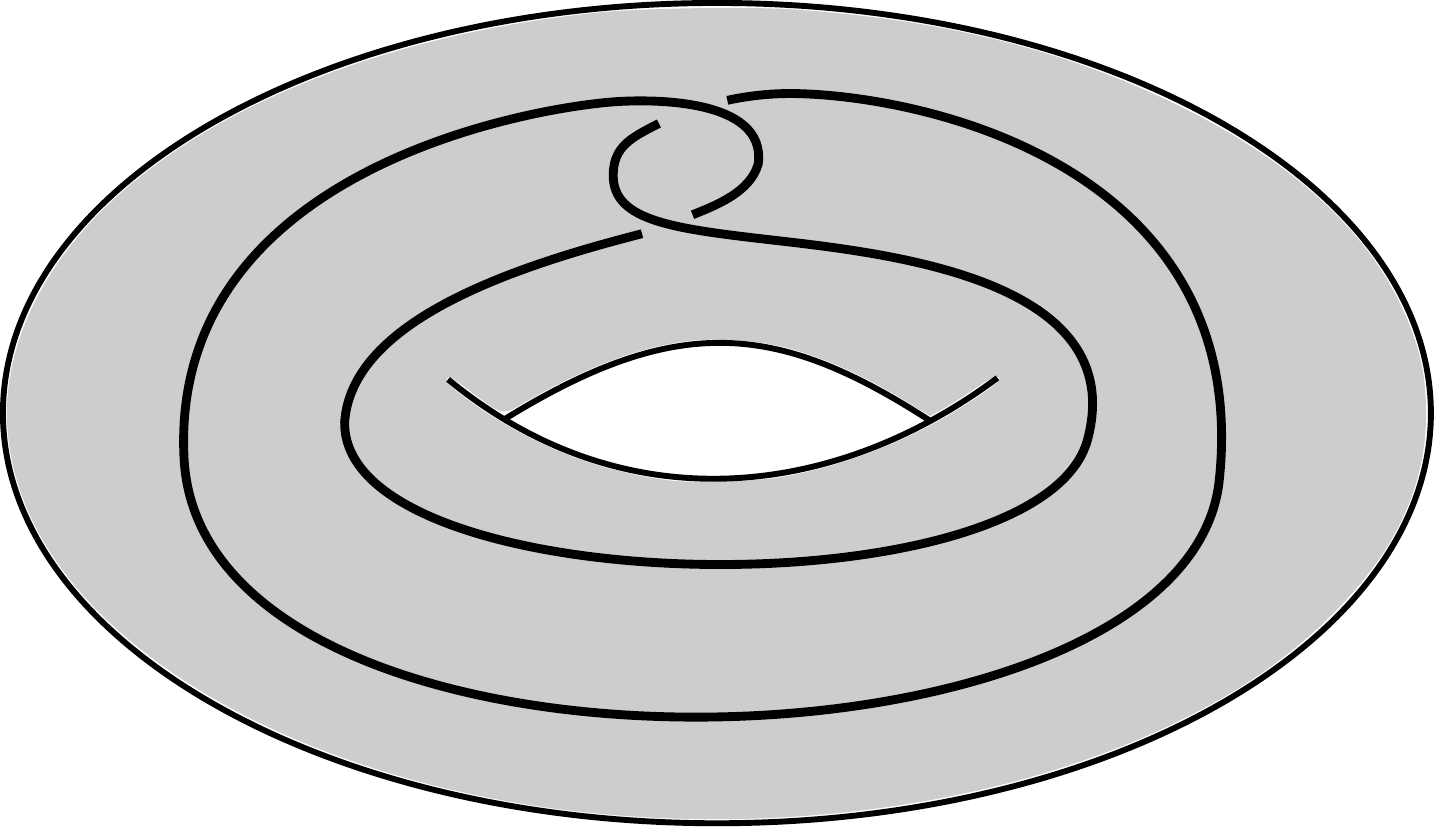}
\caption{The positive Whitehead pattern inside a solid torus.}
\label{fig:WD}
\end{figure}
Our next application pertains to Whitehead doubles of knots.  Recall that the \textit{$n$-framed positive Whitehead double} of a knot $K,$ which we denote $D_+(K,n)$ is the knot obtained by gluing in a solid torus $V$ with the positive Whitehead pattern (see Figure \ref{fig:WD}) to the exterior of $K$ such that the standard longitude of $V$ is identified with a $n$-framed longitude of $K,$ and the meridian of $V$ is glued to a meridian of $K.$

Our obstruction may be used to find:
\begin{restatable}{thm}{WD}
Let a knot $K$ be given, and suppose $a_2(K)= 0.$  If $|n|\geq 5,$ then $D_+(K,n)$ does not admit any chirally cosmetic surgeries.  If, moreover, $2\tau(K)>n$ for $n\in \lbrace-4,-2,1,4\rbrace,$ then $D_+(K,n)$ also admits no chirally cosmetic surgeries.
\end{restatable}

In the final section, we turn our attention to the case of L-space knots.  As an application of our new obstruction, we prove a partial converse to Theorem \ref{thm:rk}:

\begin{restatable}{thm}{lspace}\label{thm:lspace}
If $K$ is a nontrivial L-space knot, then $K$ does not admit any cosmetic surgeries along slopes of opposite signs.
\end{restatable}

\section{Obstructions from Heegaard Floer Homology}

\subsection{Heegaard/Knot Floer Floer homology and immersed curves}
Here we review the relevant background regarding the Heegaard Floer theoretic invariants we shall use to obstruct cosmetic surgeries.  Recall that the ``hat" version of Heegaard Floer homology assigns to a (closed, oriented) 3-manifold $M$ a graded vector space over $\mathbf{F}=\mathbf{Z}/2\mathbf{Z}$ denoted $\widehat{HF}(M).$  If $b_1(M)=0,$ then it has a $\mathbf{Z}$-grading.  Furthermore, this vector space decomposes along the Spin$^c$-structures of $M,$ which are (non-canonically) in bijection with the elements of $H_1(M;\mathbf{Z}).$  That is, for each Spin$^c$-structure $\mathfrak{s}\in Spin^c(M),$ we have a vector space $\widehat{HF}(M,\mathfrak{s})$ such that
\[
\widehat{HF}(M) = \bigoplus_{\mathfrak{s}\in Spin^c(M)} \widehat{HF}(M,\mathfrak{s}).
\]
Recall also that there is a related construction, originally due to Ozsv\'{a}th/Szab\'{o} \cite{OSzHFK} and Rasmussen \cite{Ras}, that gives invariants of (nullhomologous) knots in 3-manifolds.  Since we are concerned with knots in $S^3,$ we focus on this case.  In general, one assigns to a knot $K\subset S^3$ a graded, filtered chain complex $CFK^-(K)$ over $\mathbf{F}[U],$ where $U$ is a formal variable.  The (filtered) chain homotopy type of this chain complex is an invariant of the knot.  From this general invariant, there are a few simpler ones that can be derived; in particular, there is a ``hat" version $\widehat{CFK}(K)$ which is obtained by setting $U=0.$  The homology of (the associated graded object corresponding to) this chain complex, denoted $\widehat{HFK}(K)$ is a bigraded vector space over $\mathbf{F}.$  The two gradings are called the Alexander and Maslov gradings, and we denote the subspace of elements with Alexander grading $s$ and Maslov grading $\mu$ by $\widehat{HFK}_{\mu}(K,s).$  From this, one also has a connection to the classical Alexander polynomaial as follows:
\begin{equation}\label{eq:euler}
\Delta_K(t) = \sum_{\mu,s}  \rk \widehat{HFK}_{\mu}(K,s)(-1)^{\mu} t^s.
\end{equation}

We now turn to a review of the immersed curve versions of these Heegaard Floer theoretic invariants.  In the general setting, these immersed curve invariants are a reformulation of certain bordered invariants (that is, invariants for 3-manifolds with boundary); for our purposes, we shall use the construction for manifolds with torus boundary, first studied by Hanselman/Rasmussen/Watson \cite{HRW}.  This invariant assigns to a 3-mainfold with (parametrized) torus boundary $M$ a collection of immersed curves (with local systems) on the punctured torus $T_{\bullet}$ which encodes a certain bordered invariant (specifically, $\widehat{CFD}(M)$).  For our purposes, we will be interested in the case where $M=S^3\setminus \mathring{N}(K),$ the exterior of some knot $K.$  It is a fact that $\widehat{CFD}(M)$ is determined by $CFK^-(K)$ in this context; see Theorem 11.26 of \cite{LOT}. As a result, one may think of the immersed curve invariant of the knot exterior as encoding the Knot Floer homology of $K$ (see \cite{HRW2} for more details).  Moreover, the pairing results from the bordered theory have a particularly nice formulation in the immersed curve setting:
\begin{thm}[Theorem 2 of \cite{HRW}]\label{thm:pairing}
Given two 3-manifolds $M_1,M_2$ with (parametrized) torus boundary, let $\Gamma_i$ be the immersed curve invariant associated to each $M_i.$  If $Y=M_1\cup_{\phi}M_2$ is the closed 3-manifold obtained by gluing $M_1$ and $M_2$ along their boundaries by $\phi:\partial M_2\rightarrow \partial M_1,$ then
\[
\widehat{HF}(Y)\cong HF(\Gamma_1,\overline{\phi}(\Gamma_2)),
\]
where $HF$ denotes the Lagrangian Floer homology on the punctured torus $T_{\bullet},$ and $\overline{\phi}$ is the map $\phi$ composed with the elliptic involution.
\end{thm}
As the solid torus has a particularly simple bordered invariant, in particular, the immersed curve is a single embedded (non-nullhomotopic) circle, Theorem \ref{thm:pairing} gives a useful geometric way to compute Heegaard Floer homology of Dehn surgeries.  More precisely, if $\Gamma$ is the immersed curve invariant of a knot $K,$ then $\widehat{HF}(S^3_K(\frac{p}{q}))\cong HF(\Gamma,L_{p/q}),$ where $L_{p/q}$ is the embedded circle of ``slope" $\frac{p}{q}$ on the (punctured) torus with respect to the standard meridian-longitude parametrization.
\begin{rmk}\label{rmk:tight}
In order to compute the Lagrangian Floer homology of (immersed) curves, we shall generally assume the curves have been placed in minimal position so that there are no immersed bigons bounded by the paired curves, and hence the rank is simply the number of intersection points as the differential is trivial.  The only possible exception would be if two parallel curves were being paired, in which case minimal position would result in an immersed annulus being bounded by the pair.  However, since we are considering pairings with curves of the form $L_{p/q},$ we shall see below that the only way this can occur is if $p=0,$ which we do not consider as there can be no cosmetic surgeries involving slope zero.  Moreover, for the same reason, we need not worry about non-trivial local systems.  More precisely, interpreting a non-trivial local system on a curve as multiple copies of the curve along with some extra ``train-track" connections, the only time this affects the intersection number is if there are parallel curves being paired.
\end{rmk}
\begin{figure}
\centering
\includegraphics[width=.25\textwidth]{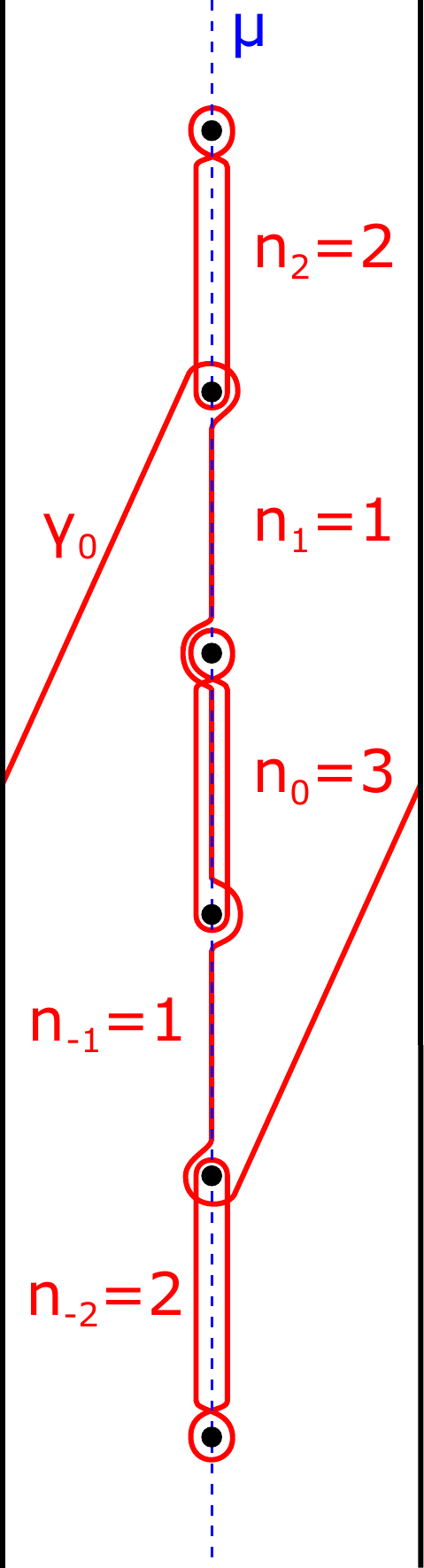}
\caption{An example of immersed curves in $\overline{T}_{\bullet}$ corresponding to some knot-like complex.  The left and right edges are identified.  We can see that this knot would have genus 3, tau 2 and epsilon 1.}
\label{fig:ex}
\end{figure}
For our purposes, it will be more useful to consider lifts of the immersed curves to the cylindrical cover $\overline{T}_{\bullet}$ of the punctured torus. More precisely, consider the immersed curve invariant $\Gamma$ coming from the exterior of a knot $K.$ We identify $T$ with the boundary of the knot exterior and consider the covering space $\overline{T}$ associated with the subgroup of $\pi_1(T)$ generated by the (Seifert) longitude of $K.$  Restricting this cover to the punctured torus $T_{\bullet},$ we have the covering space $\overline{T}_{\bullet},$ which is the infinitely-punctured cylinder.  Following \cite{Hans}, we will parametrize this as $ S^1\times \mathbf{R} \setminus \lbrace*\rbrace \times (\mathbf{Z}+\frac{1}{2});$ that is, the punctures are all arranged with ``heights" at the half-integers.  We distinguish a special lift $\mu$ of the meridian, which is the ``vertical" line passing through the punctures.  The punctures cut $\mu$ into a union of disjoint intervals of the form $(i-\frac{1}{2},i+\frac{1}{2})$ for each $i\in \mathbf{Z}.$  We refer to points in such an interval as being at \textit{level $i.$}  There is a standard lift of $\Gamma$ to $\overline{T}_{\bullet},$ which we denote as $\overline{\Gamma}.$  The symmetries of Knot Floer homology imply that, after a suitable homotopy of the immersed curves, $\overline{\Gamma}$ has a $180^{\circ}$ rotational symmetry, and we take the symmetry point to be the ``origin" (that is, the point on $\mu$ at height 0).  Another important property of $\overline{\Gamma}$ is that there is a distinguished component, denoted $\gamma_0,$ which is the unique curve that ``wraps around" the cylinder $\overline{T}_{\bullet}.$  All other components of $\overline{\Gamma}$ may be homotoped to lie in an arbitrarily small neighborhood of $\mu.$ We shall usually assume that all the immersed curves have undergone an appropriate homotopy so that they are ``pulled tight" with respect to some $\epsilon$-neighborhood of the punctures.  Such curves will generically pair minimally with ``lines" which are lifts of $L_{p/q},$ as in Remark \ref{rmk:tight}.  This will be of interest to us due to the following consequence of Proposition 46 of \cite{HRW} (see also Theorem 14 of \cite{Hans}):
\begin{thm}\label{thm:pairS}
Let $K \subset S^3$ be a knot and let $\Gamma \subset T_{\bullet}$ be its immersed curve invariant, with $\overline{\Gamma} \subset \overline{T}_{\bullet}$ the standard lift. Given $\frac{p}{q}\in \mathbf{Q}\cup \{\infty\},$ for each $\mathfrak{s}\in Spin^c(S^3_K(\frac{p}{q})),$ there exists a (unique homotopy class of) lift $\overline{L^{\mathfrak{s}}}_{p/q}$ of $L_{p/q}$ such that:
\[
\widehat{HF}\left(S^3_K\left(\frac{p}{q}\right),\mathfrak{s}\right) \cong HF\left(\overline{\Gamma},\overline{L^{\mathfrak{s}}}_{p/q}\right).
\]
\end{thm}
Here $HF$ is the Lagrangian Floer homology of immersed curves in the (infinitely) punctured cylinder $\overline{T}_{\bullet}.$ As mentioned above, the rank of this is computed in practice as the minimal intersection number between the two (multi)curves.
Alternatively, we may slightly perturb the ``pulled tight" position to make all segments vertical outside of $\epsilon$-neighborhoods of the punctures, except for the part of $\gamma_0$ that wraps around.  With the curves in this position, for each $i\in \mathbf{Z},$ we denote by $n_i(K)$ the number of vertical segments at level $i.$ Note that this is an invariant of the knot $K,$ as indeed one could define it as the minimal number of such vertical segments among all representatives of the homotopy class of $\overline{\Gamma}.$  See Figure \ref{fig:ex} for an example of a knot-like immersed curve.  Several invariants can be extracted from such immersed curve diagrams:
\begin{itemize}
\item The Seifert genus $g(K)$ is the maximum level at which some curve of $\overline{\Gamma}$ intersects $\mu.$
\item The tau invariant $\tau(K)$ is the level at which the distinguished curve $\gamma_0$ first intersects $\mu$ (when following the curve from the ``left").
\item The epsilon invariant $\epsilon(K)$ is given by the direction $\gamma_0$ ``turns" after the aforementioned first intersection with $\mu:$ 1 if downward, -1 if upward, and 0 if there is no turning.
\end{itemize}
In the ``pulled tight" configuration, we can thus see that the unique non-vertical segment (which is part of $\gamma_0$) will have a slope of $2\tau(K)-\epsilon(K).$  Let us define $\V(K)$ to be the total number of vertical segments of $\overline{\Gamma};$ that is,
\[
\V(K) = \sum_i n_i(K).
\]
In the punctured torus, we see that, when pulled tight, $\Gamma$ effectively has the form of $\V(K)$ copies of the meridian circle along with a circle of slope $2\tau(K)-\epsilon(K).$  Since the ``pulled tight" configuration will have minimal intersection number with any (transverse, non-nullhomologous) embedded circle, we find the following formula for the rank of the Heegaard Floer homology of knot surgeries:
\begin{equation}\label{eq:rk}
\rk \widehat{HF}\left(S^3_K\left(\frac{p}{q}\right)\right) = |p-q(2\tau(K)-\epsilon(K))| + |q|\V(K).
\end{equation} 
\begin{proof}[Note:]\let\qed\relax
This formula needs to be modified in the case that $\frac{p}{q}=0$ and $\gamma_0$ is a horizontal curve (equivalently, $\epsilon(K)=0$); the actual rank is two more than what this formula predicts.  However, for our purposes, this case will not appear.
\end{proof}
\subsection{The main obstruction}

\begin{figure}
\centering
\includegraphics[width=.25\textwidth]{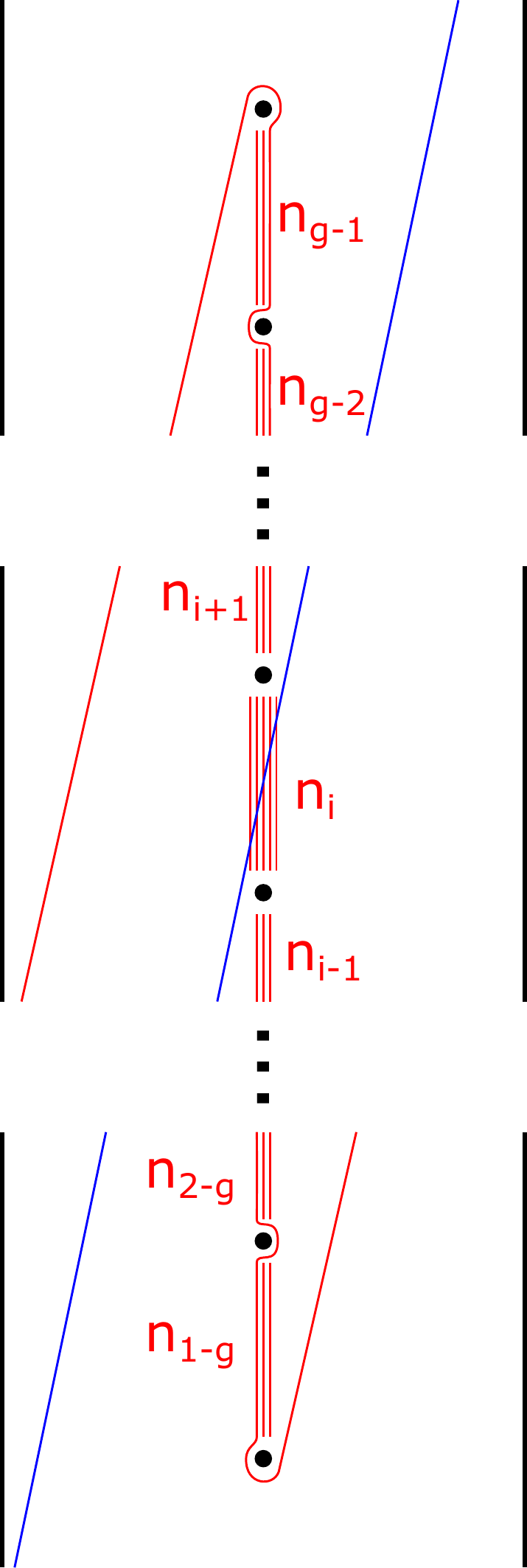}
\hspace{.25\textwidth}
\includegraphics[width=.25\textwidth]{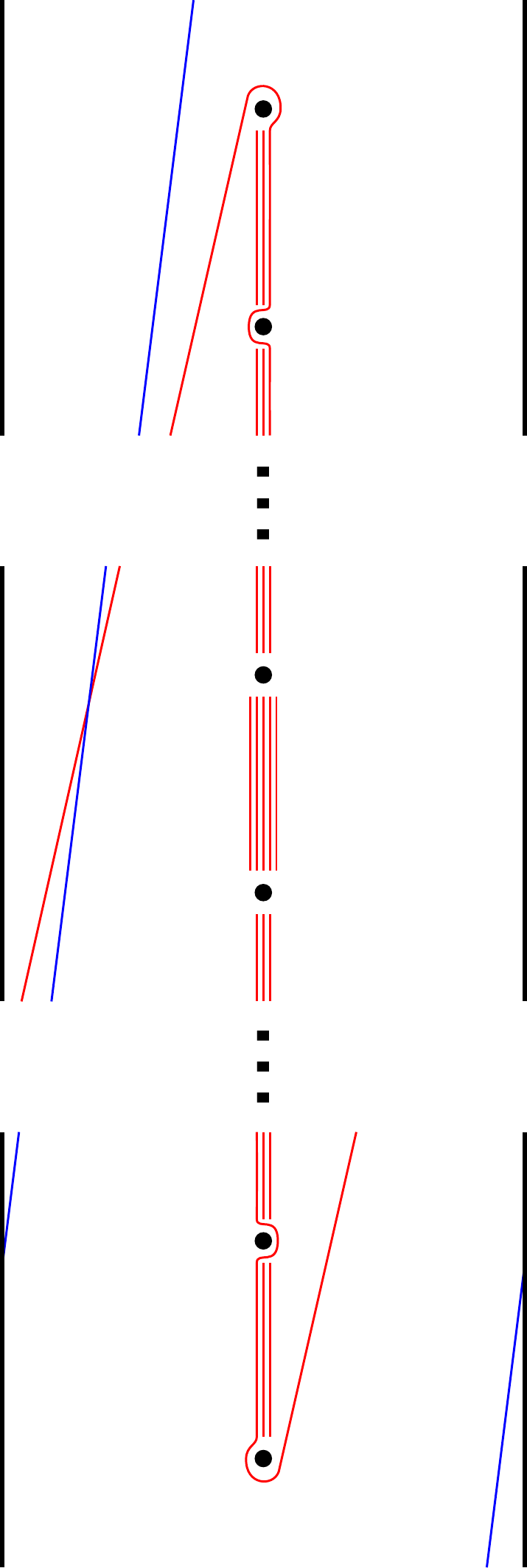}
\caption{Schematics of the kinds of immersed curve pairings occurring in the proof of Lemma \ref{lem:pos}.}
\label{fig:pos_int}
\end{figure}

\begin{lem}\label{lem:pos}
Let $K$ be a knot with $\tau(K)=g(K)>0.$  If $\frac{p}{q}>2g(K)-1$ then $\rk\widehat{HF}(S^3_K(\frac{p}{q}),\mathfrak{s})\leq \max_i n_i(K)$ for each Spin$^c$ structure $\mathfrak{s}\in Spin^c(S^3_K(\frac{p}{q})).$
\end{lem}
\begin{proof}
Since $\tau(K)=g(K)>0,$ we must have $\epsilon(K)=1,$ which in the immersed curve setting, means that $\gamma_0 \subset \overline{T}$ curves downward after reaching the midline $\mu$.  Hence, the slope of the non-vertical segment of $\gamma_0$ is $2\tau(K)-1=2g(K)-1.$  Let $\mathfrak{s}\in Spin^c(S^3_K(\frac{p}{q}))$ be given.  Then by Theorem \ref{thm:pairS}, $\rk\widehat{HF}(S^3_K(\frac{p}{q}),\mathfrak{s})$ is equal to the minimal intersection number between $\overline{\Gamma}$ (the lift of the immersed curve invariant of $K$) and the lift  $\overline{L^\mathfrak{s}}_{p/q}\subset \overline{T}$ of $L_{p/q}$ corresponding to $\mathfrak{s}.$ As $\frac{p}{q}>2g(K)-1,$ we see that $\overline{L^\mathfrak{s}}_{p/q}$ can intersect $\mu$ between $-g+\frac{1}{2}$ and $g-\frac{1}{2}$ at most once.  We consider the two possibilities in turn: in the case where there is such an intersection, say at height $i,$ the number of intersection points between $\overline{\Gamma}$ and $\overline{L^\mathfrak{s}}_{p/q}$ will be $n_i(K)$ within this height range. Since $\frac{p}{q}>2g(K)-1,$ we also see that there will be no intersection with the non-vertical segment of $\gamma_0.$   Moreover, since $\overline{\Gamma}$ is supported between heights of $-g+\frac{1}{2}+\epsilon$ and $g-\frac{1}{2}-\epsilon,$ there can be no further intersections; therefore, the number of intersections is at most $\max_i n_i(K)$ in this situation. See the left of Figure \ref{fig:pos_int} for an illustration of this first case.  In the second case, we consider the possibility that $\overline{L^\mathfrak{s}}_{p/q}$ does not intersect $\mu$ within the interval $[-g+\frac{1}{2},g-\frac{1}{2}].$  Here there will be no intersections with any vertical segments of $\overline{\Gamma},$ and so the only intersections that may occur are with the segment of slope $2g(K)-1.$  Since $\frac{p}{q}>2g(K)-1,$ this can happen at most once (see the right side of Figure \ref{fig:pos_int} for an illustration).  Of course, since $\tau(K)>0$ by hypothesis, $n_0(K)\geq 1$ and so we also see that the number of intersection points is no more than $\max_i n_i(K)$ in this case as well.
\end{proof}

\begin{figure}
\centering
\includegraphics[width=.25\textwidth]{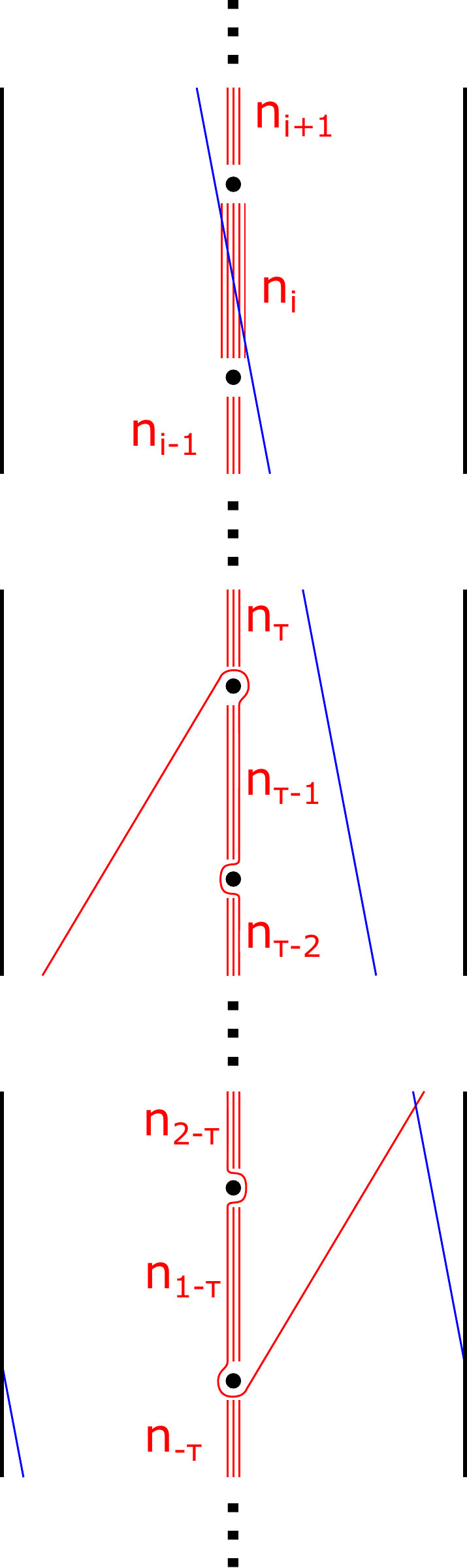}
\hspace{.25\textwidth}
\includegraphics[width=.25\textwidth]{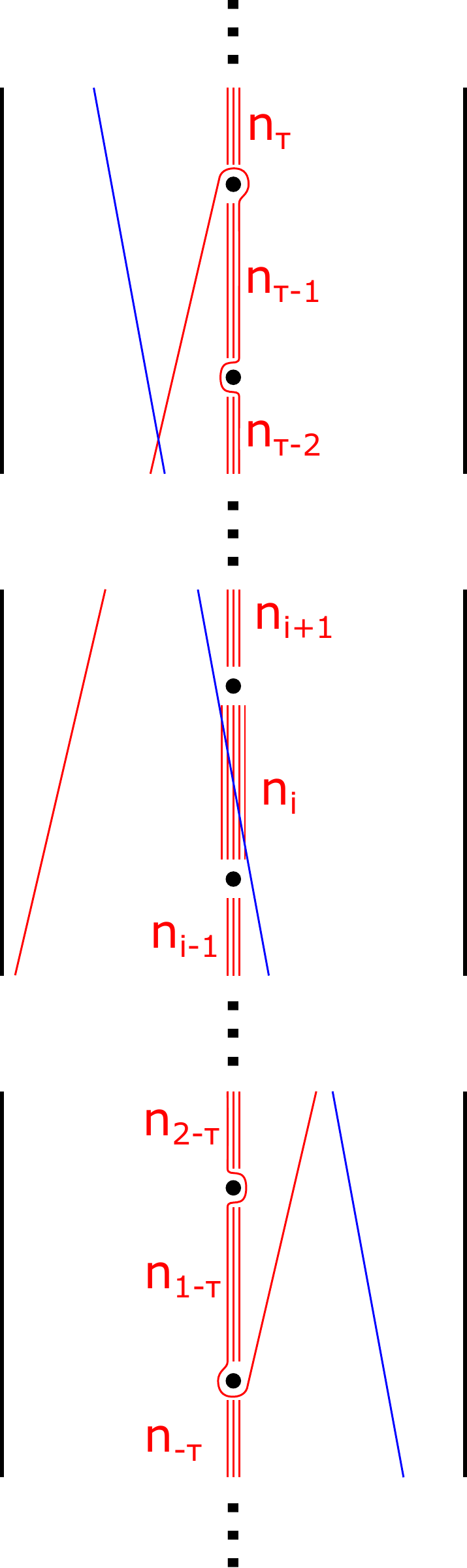}
\caption{Schematics of the kinds of immersed curve pairings occurring in the proof of Lemma \ref{lem:neg}.}
\label{fig:int_neg}
\end{figure}

\begin{lem}\label{lem:neg}
Let $K$ be a knot with $\tau(K)>0.$  If $\frac{p}{q}<0$ then there exists an $\mathfrak{s}\in Spin^c(S^3_K(\frac{p}{q}))$ such that $\rk\widehat{HF}(S^3_K(\frac{p}{q}),\mathfrak{s})\geq \max_i n_i(K)+1.$
\end{lem}
\begin{proof}
Let $i$ be such that $n_i(K)=\max_j n_j(K).$  There exists a lift $\overline{L}\subset \overline{T}_{\bullet}$ of $L_{p/q}$ that intersects $\mu$ at level $i.$ By Theorem \ref{thm:pairS}, we choose $\mathfrak{s}$ to be the Spin$^c$ structure corresponding to this lift; that is, in the notation as above, $\overline{L}=\overline{L^\mathfrak{s}}_{p/q}.$  We consider two separate cases.  If $|i|\geq\tau(K),$ then since $\frac{p}{q}<0$ and $\tau(K)>0,$ the line $\overline{L^\mathfrak{s}}_{p/q}$ must intersect with the non-vertical segment of $\gamma_0.$  This can be seen as follows: assuming without loss of generality that $i>0,$ we follow the line $\overline{L^\mathfrak{s}}_{p/q}$ from an intersection at level $i$ until the point at which it reaches height 0.  If it has wrapped at least halfway around the cylinder $\overline{T},$ then it will have intersected with the ``positive" non-vertical segment of $\gamma_0.$  Otherwise, at the point where it wraps halfway, it will have a negative height and hence will have intersected with the ``negative" part of the non-vertical segment of $\gamma_0$ (see the left side of Figure \ref{fig:int_neg} for an illustration of this case).  Either way, along with the $n_i$ intersections with vertical segments, there is at least one additional intersection point.  In the other case we have $|i|<\tau(K).$  If at level $i$ the intersection of $\overline{L^\mathfrak{s}}_{p/q}$ with $\mu$ occurs at a positive height, then since $\frac{p}{q}<0$ and $\tau(K)>0,$ there must be an additional intersection point between the aforementioned intersection point and the point where $\overline{L^\mathfrak{s}}_{p/q}$ has wrapped halfway around the cylinder upwards and to the left (see the right side of Figure \ref{fig:int_neg} for an illustration).  An analogous argument holds if the intersection with $\mu$ occurred at a negative height.  In both cases, we see that there are at least $n_i(K)+1$ intersection points.  Since $n_i(K)$ was chosen to be maximal, the desired result follows.
\end{proof}

We now prove our main obstruction:
\main*
\begin{proof}
For the first part, if $\frac{p}{q}>2g(K)-1,$ Lemmas \ref{lem:pos} and \ref{lem:neg} imply that $S^3_K(\frac{p}{q})$ and $S^3_K(\frac{p}{q'})$ cannot be homeomorphic as there can be no isomorphism between their Heegaard Floer homologies that respects the Spin$^c$ decomposition.
As $\tau(K)=g(K)>0,$ we have a non-vertical segment of slope $2g(K)-1.$ From the rank formula \eqref{eq:rk}, we have that:
\[
|p-q(2g(K)-1)|+\V(K)q = p-q'(2g(K)-1)-\V(K)q'.
\]
As $\frac{p}{q}\leq 2g(K)-1,$ this becomes:
\begin{align*}
q(2g(K)-1)-p+\V(K)q &= p-q'(2g(K)-1)-\V(K)q' \\
(q+q')(2g(K)-1+\V(K)) &= 2p
\end{align*}

\end{proof}

\subsection{Combining other obstructions}\label{sec:combo}
Here we briefly recall some facts about finite type invariants (also called Vassiliev invariants) for knots.  Suppose a real-valued knot invariant $v$ can be extended to an invariant of \textit{singular knots} (i.e., knots with possibly finitely many points of self-intersection) in a way that satisfies the following:
\[
v(K_{\bullet}) = v(K_+)-v(K_-)
\]
whenever the (singular) knots $K_+$, $K_-$, and $K_{\bullet}$ differ locally near a crossing/self-intersection as in Figure \ref{fig:skein}.  Then $v$ is said to be a finite type invariant of order $n$ if moreover $v(K)=0$ whenever $K$ has at least $n+1$ self-intersection points.

\begin{figure}
\centering
\includegraphics[width=.75\textwidth]{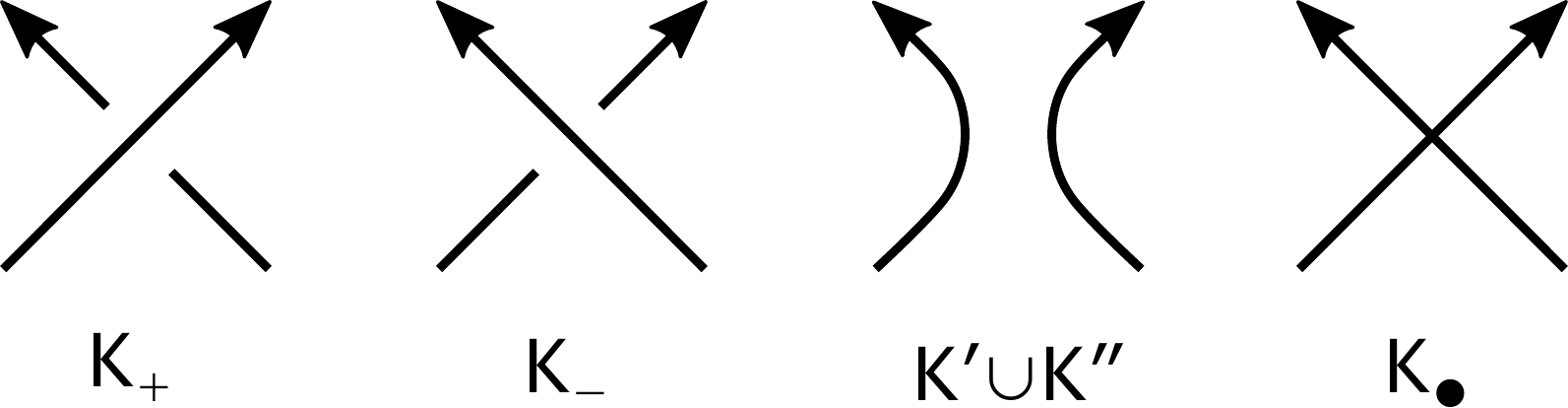}
\caption{The diagrams of $K_+$, $K_-$, $K'\cup K''$, and $K_{\bullet}$ differ only locally near a crossing.}
\label{fig:skein}
\end{figure}

For example, recall the Conway polynomial of a knot, which is related to the Alexander polynomial in the following way:
\[
\nabla_K(z) = \Delta_K(t)|_{z=t^{1/2}-t^{-1/2}}
\]
For a knot $K$, the Conway polynomial will have the form:
\[
\nabla_K(z) = 1 + \sum_{j=1}^{n}a_{2j}(K)z^{2j}
\]
It is a fact that the coefficent $a_{2j}(K)$ is a finite type invariant of order $2j$ for each $j$.  For instance, $a_2(K) = \frac{1}{2}\Delta_K''(1).$

The other finite type invariant that will be of interest to us is $v_3$.  This is a third-order invariant, which may be defined as:
\[
v_3(K) = -\frac{1}{36}V'''_{K}(1) -\frac{1}{12}V''_{K}(1),
\]
where $V_K(t)$ is the Jones polynomial of $K$.

It is known that the Conway polynomial satisfies the following skein relation:
\begin{equation}\label{eq:conwaySkein}
\nabla_{K_+}(z)-\nabla_{K_-}(z)=z\nabla_{K'\cup K''}(z).
\end{equation}

The relevance of these finite type invariants to cosmetic surgery is seen in the following result obtained by Ito using the degree 2 part of the LMO invariant:
\begin{thm}[Corollary 1.3 of \cite{Ito}]\label{thm:ft}
Let $K$ be a knot and suppose $S^3_{p/q}(K) \cong \pm S^3_{p/q'}(K)$ for some $q\neq q'.$ Then,
\[
2pv_3(K) = (7a_2(K)^2 - a_2(K)-10a_4(K))(q+q')
\]
\end{thm}
\begin{rmk}
Ito uses a slightly different definition for $v_3$ than is used here; in particular, he normalizes it to take the value $\frac{1}{4}$ on the right-handed trefoil instead of 1 as in our convention.
\end{rmk}

Combining Theorems \ref{thm:main} and \ref{thm:ft} we find the following obstruction to the existence of chirally cosmetic surgeries with opposite signs:
\begin{cor}\label{cor:combo}
Let $K$ be a knot satisfying $\tau(K)=g(K)$ and $v_3(K)\neq0$.  If $S^3_{p/q}(K) \cong - S^3_{p/q'}(K)$ for some $p,q>0,q'<0$ then
\[
v_3(K)(\V(K)+2g(K)-1) = 7a_2(K)^2 - a_2(K)-10a_4(K).
\]
\end{cor}
\begin{proof}
Substition gives:
\[
v_3(K)(q+q')(2g(K)-1+\V(K)) = (7a_2(K)^2 - a_2(K)-10a_4(K))(q+q').
\] 
Notice that Theorem \ref{thm:main} implies that $q+q' \neq 0$ in this setting, and so the result follows.
\end{proof}

Note that if $v_3(K)\neq 0,$ the obstruction above may be written as:
\[
\V(K)+2g(K)-1 = \frac{7a_2(K)^2 - a_2(K)-10a_4(K)}{v_3(K)}.
\]
The quantity on the left-hand side is Heegaard-Floer theoretic, while that on the right-hand side is a combination of finite type invariants, so for a generic knot, one would not expect these two to coincide.  Hence, this corollary provides some reason to believe in a negative answer to Question \ref{ques:opp}; the conditions $\tau(K)=g(K)$ and $v_3(K)\neq 0$ are both sufficient (although certainly not necessary) for $K$ to be non-amphichiral.  Morally, one might interpret this as suggesting that knots which are ``very far" from being amphichiral will not admit chirally cosmetic surgeries along slopes of opposite signs. 

\section{Applications}

\subsection{Pretzel knots}

The goal of this section is to prove:
\pretz*

Here we use $s_{n,m}$ as a shorthand for $s_n(k_1,\dots,k_m),$ the $n$th elementary symmetric polynomial in $k_1,\dots,k_m$ given by:
\[
s_{n,m}=\sum_{\substack{P\subset\{1,\dots,m\} \\ |P|=n}} \prod_{j\in P} k_j
\]
Let us here record some properties of the elementary symmetric polynomials, which are straightforward consequences of their definition.
\begin{lem}\label{lem:sym}
Let $s_{n,m}$ denote the $n$th elementary symmetric polynomial in $k_1,\dots,k_m$.  Then, the following hold:
\begin{itemize}
\item If $k_j=0$ for all $j>N$, then whenever $n\geq N$, $s_{n,m} = s_{n,N}$
\item $s_{n,m+1} = s_{n,m} + k_{m+1}s_{n-1,m}$
\end{itemize}
\end{lem}
\begin{rmk}
By convention, we take $s_{0,m}=1$, and if $n>m$, $s_{n.m}=0$.
\end{rmk}

In the general case where $K=K(k_1,\dots,k_{2g+1})$, we will need the following formulae for $a_2(K)$ and $v_3(K)$.
\begin{lem}[Lemma 2.2 of \cite{CCP}]\label{lem:av}
Let $K=K(k_1,\dots,k_{2g+1})$.  Then
\begin{align*}
a_2(K)&=\frac{1}{2}g(g+1) + g s_{1,2g+1} + s_{2,2g+1}
\\
v_3(K)&=\frac{1}{2}\left(\frac{g(g+1)(2g+1)}{3} + g(2g+1)s_{1,2g+1} + g s_{1,2g+1}^2 + 2g s_{2,2g+1} + s_{1,2g+1}s_{2,2g+1} + s_{3,2g+1} \right)
\end{align*}
\end{lem}
\begin{rmk}
The formula for $v_3$ here has the opposite sign from that of \cite{CCP} because here we are considering the mirror image of the knots considered in that paper.  Of course, the Alexander polynomial is unchanged under mirroring, and so $a_2$ is the same as well.
\end{rmk}

\begin{figure}
\centering
\begin{subfigure}{.75\textwidth}
\centering
\raisebox{-.5\height}{
\includegraphics[width=.4\textwidth]{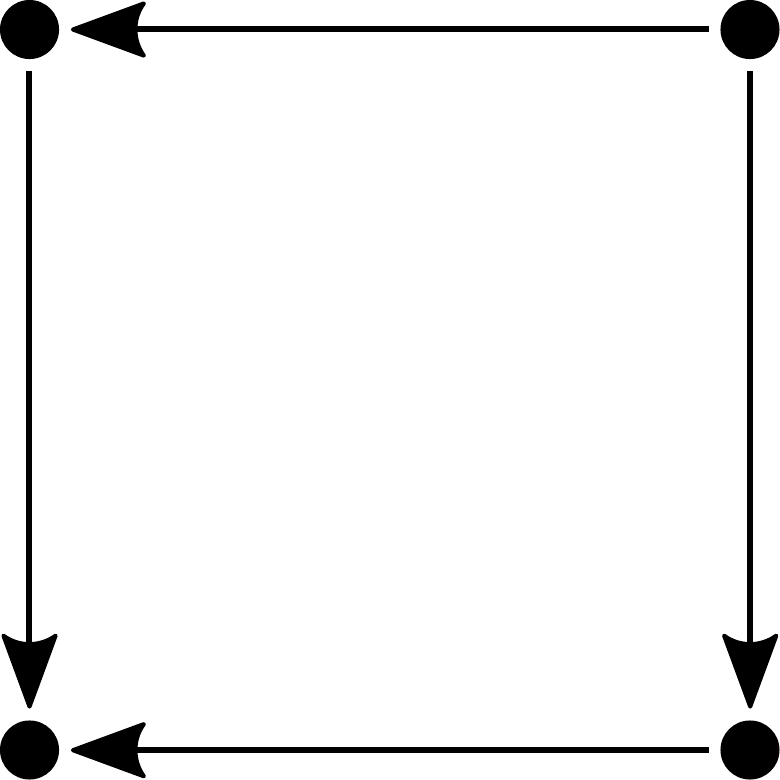}}
%\caption{A $1\times 1$ square subcomplex}
%\end{subfigure}
\hfill
%\begin{subfigure}{.45\textwidth}
%\centering
\raisebox{-.5\height}{
\includegraphics[width=.2\textwidth]{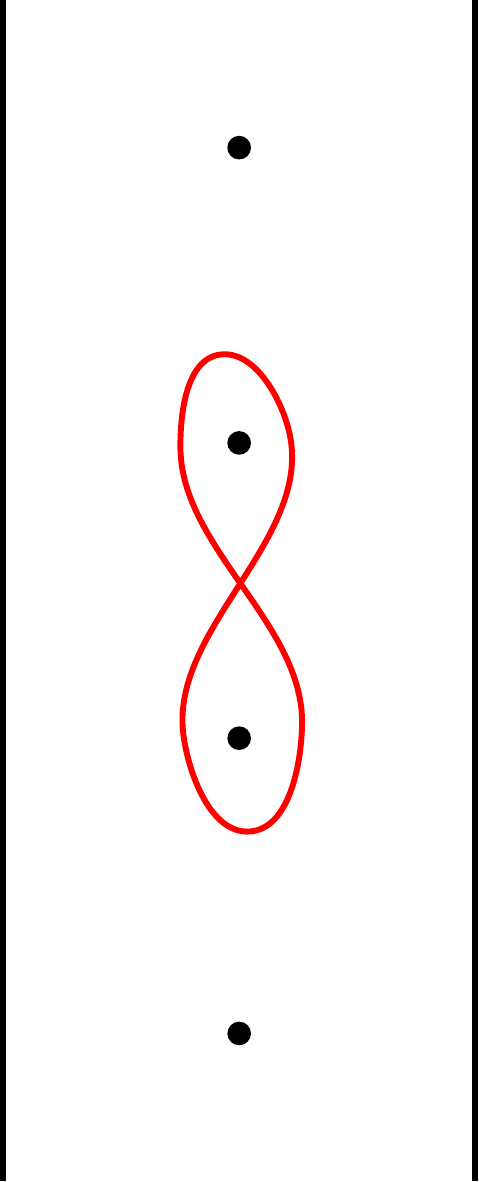}}
\caption{A $1\times 1$ square in the Knot Floer complex (left) and its corresponding immersed curve invariant (right), a simple figure-eight curve.}
\end{subfigure}
\\
\begin{subfigure}{.85\textwidth}
\centering
\raisebox{-.5\height}{
\includegraphics[width=.65\textwidth]{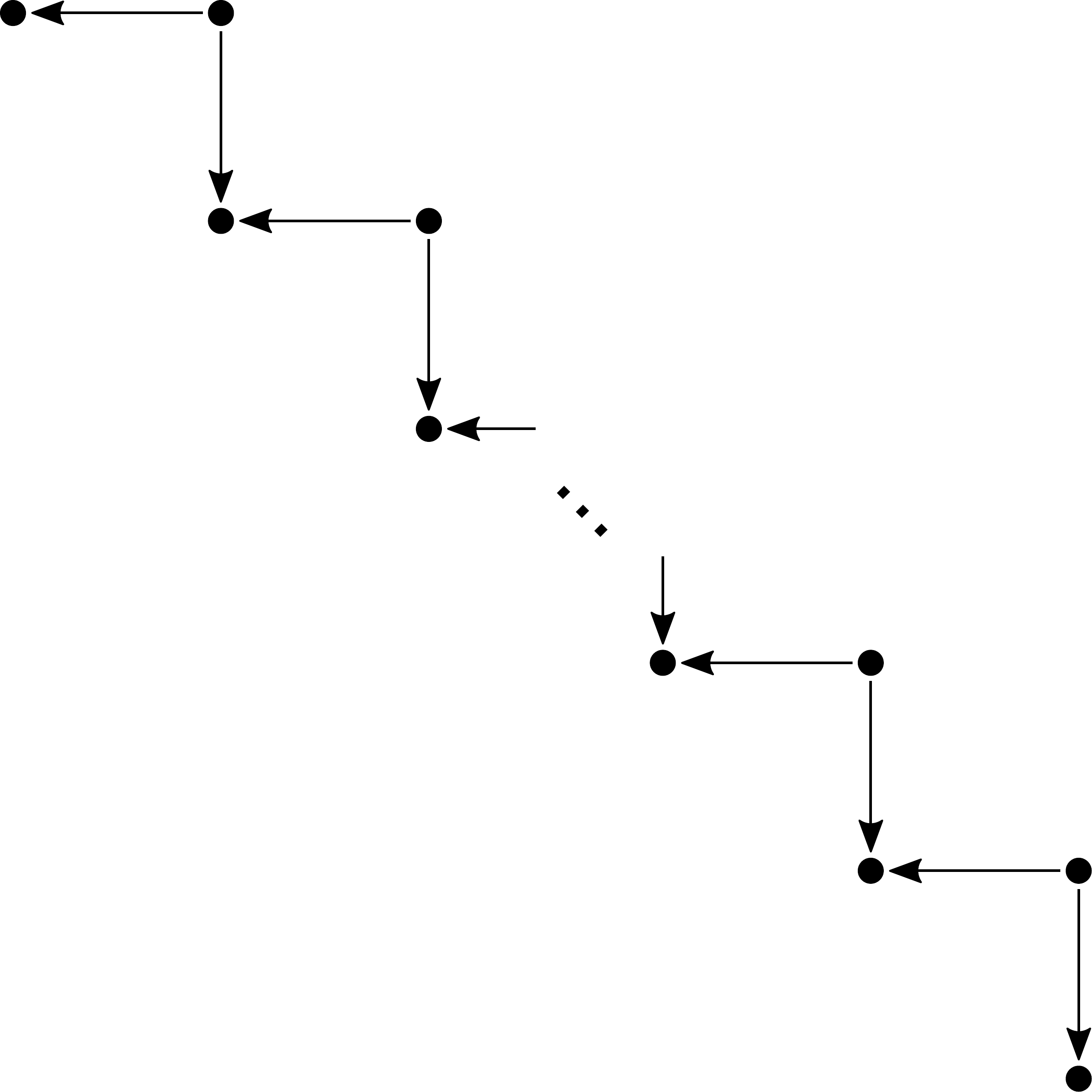}}
%\caption{A ``staircase" subcomplex with length 1 arrows}
%\end{subfigure}
\hfill
%\begin{subfigure}{.25\textwidth}
%\centering
\raisebox{-.5\height}{
\includegraphics[width=.15\textwidth]{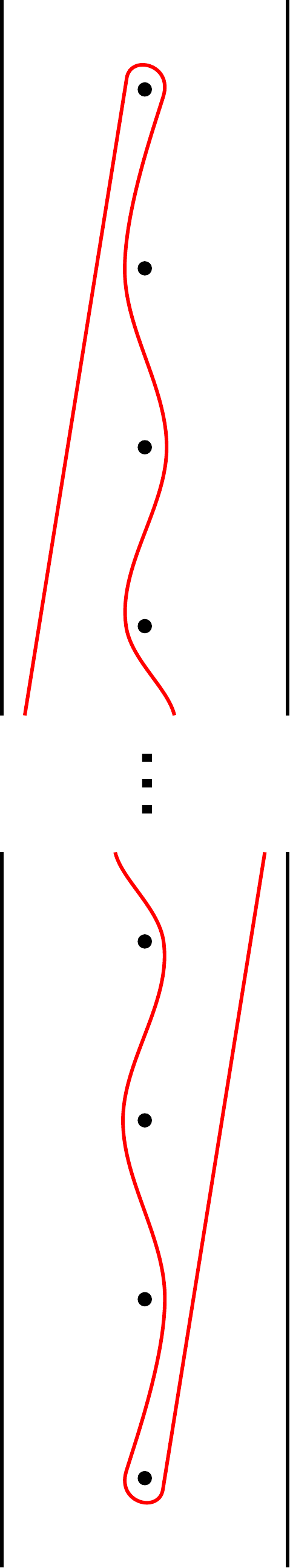}}
\caption{A ``staircase" subcomplex with length 1 arrows (left) and the corresponding immersed curve invariant (right); this is the $\tau(K)>0$ case.}
\end{subfigure}
\caption{The two kinds of Knot Floer subcomplexes a thin knot may have along with their corresponding immersed curve invariants.}
\label{fig:alt_cpx}
\end{figure}

We wish to compute $\V(K)$ for this family of knots.  Since these knots are alternating (and hence homologically thin), it is known that their (reduced) Knot Floer complex may be put in a special form.  In particular, there exists a reduced basis such that the only ``arrows" are of length one, so that the complex consists of $1\times 1$ ``squares" and one ``staircase" subcomplex; see Lemma 7 of \cite{Pet}.  These correspond, in the immersed curve formulation, to simple (i.e., height 1) figure-eight curves and a $\gamma_0$ that ``weaves" through adjacent punctures, respectively (see Figure \ref{fig:alt_cpx}).  Each figure-eight component contributes 4 to the rank of $\widehat{HFK}(K)$ and 2 to $\V(K).$ On the other hand, $\gamma_0$ for this knot will have height equal to $\tau(K).$ Hence, if we put $F$ equal to the number of simple figure-eight components, we have that $\rk\widehat{HFK}(K)=4F+2\tau(K)+1$ and $\V(K)=2F+2\tau(K)-1$ (this assumes $\tau(K)>0$).  So we find that:
\[
\V(K)=\frac{1}{2}\rk\widehat{HFK}(K)+\tau(K)-\frac{3}{2}.
\]
Furthermore, since $K$ is thin for each Alexander grading $s,$ $\widehat{HFK}(K,s)$ is supported in Maslov grading $\mu=s+\delta$ for some fixed $\delta$ which depends only on the knot $K.$ Hence, we see by \eqref{eq:euler} that
\begin{align*}
\Delta_K(-1)&=\sum_{\mu,s}(-1)^{\mu}(-1)^s \rk\widehat{HFK}(K,s)\\
&= \sum_{s}(-1)^{2s+\delta}\rk\widehat{HFK}(K,s)\\
&=\pm \sum\rk\widehat{HFK}(K,s)\\
&=\pm \rk\widehat{HFK}(K).
\end{align*}
Therefore, $\rk\widehat{HFK}(K)=|\Delta_K(-1)|.$  This latter number is an invariant also known as the \textit{determinant} of the knot $K$ (denoted $\det(K)$).  It follows from the definition that $\det(K)=|\det(A+A^T)|$ whenever $A$ is any matrix representing the Seifert form of $K.$  We thus have shown:
\begin{lem}\label{lem:vert}
If $K$ is an alternating (or more generally, a thin) knot with $\tau(K)>0,$ then
\[
\V(K)=\frac{1}{2}\det(K)+\tau(K)-\frac{3}{2}.
\]
\end{lem}\qed
\begin{rmk}
The case where $\tau(K)<0$ may be obtained from this by considering the mirror image of $K.$  The case $\tau(K)=0$ may also be deduced from the above methods to obtain: $\V(K) = \frac{1}{2}(\det(K)-1).$
\end{rmk}

In light of this, we now turn to the computation of the determinant of the knot $K=K(k_1,\dots,k_{2g+1}).$  Applying Seifert's algorithm to a diagram such as the one in Figure \ref{fig:Pknot}, one finds that the Seifert form for $K$ may be given by the matrix $-A_{2g},$
%\[
%A_{2g}=-\left(\begin{matrix}
%k_1+k_2+1 & k_2 & 0 & \dots & 0 & 0 \\
%k_2+1 & k_2+k_3+1 & k_3 & \dots & 0 & 0 \\
%0 & k_3+1 & k_3+k_4+1 & \dots & 0 & 0 \\
%\vdots & \vdots & \vdots & \ddots & \vdots & \vdots \\
%0 & 0 & 0 & \dots & k_{2g-1}+k_{2g}+1 & k_{2g} \\
%0 & 0 & 0 & \dots & k_{2g}+1 & k_{2g}+k_{2g+1}+1
%\end{matrix}\right)
%\]
where $A_n$ is the $n\times n$ matrix of the form:
\[
(A_n)_{ij} = \begin{cases}
k_i+k_{i+1}+1 & i=j \\
k_j & j=i+1 \\
k_i+1 & i=j+1 \\
0 & \text{otherwise}
\end{cases}
\]
We wish to compute $|\det(-A_{2g}-A_{2g}^T)|=\det(A_{2g}+A_{2g}^T).$  For the sake of notation, we put $Q_n = A_n+A_n^T$ and $p_i=2k_i+1$ for each $i.$  With this notation, we find that:
\[
(Q_n)_{ij}=\begin{cases}
p_i+p_{i+1} & i=j\\
p_i & i=j+1\\
p_j & j=i+1\\
0 & \text{otherwise}
\end{cases}
\]
\textbf{Claim:} $\det(Q_n)=s_{n,n+1}(p_1,\dots,p_{n+1}).$
\begin{proof}
This is seen by (strong) induction: the case $n=1$ is clear (for $n=0$ we interpret this as the empty product which gives a value of 1).  For $n\geq 1,$ we see by expanding along the last row and applying Lemma \ref{lem:sym} that:
\begin{align*}
\det(Q_n) &= (p_n+p_{n+1})\det(Q_{n-1})-p_n^2\det(Q_{n-2})\\
&=p_{n+1} s_{n-1,n}(p_1,\dots,p_n)+p_n s_{n-1,n}(p_1,\dots,p_n)-p_n(p_n s_{n-2,n-1}(p_1,\dots,p_{n-1})\\
&=p_{n+1} s_{n-1,n}(p_1,\dots,p_n) + p_n(s_{n-1,n}(p_1,\dots,p_n)-p_n s_{n-2,n-1}(p_1,\dots,p_{n-1}))\\
&=p_{n+1} s_{n-1,n}(p_1,\dots,p_n) + p_n s_{n-1,n-1}(p_1,\dots,p_{n-1})\\
&=p_{n+1} s_{n-1,n}(p_1,\dots,p_n) + s_{n,n}(p_1,\dots,p_n)\\
&=s_{n,n+1}(p_1,\dots,p_{n+1})
\end{align*}
\end{proof}
Now, substituting $2k_i+1$ for $p_i$ we observe that:
\begin{align*}
s_{n,n+1}(2k_1+1,\dots,2k_{n+1}+1) &= \sum_{i=1}^{n+1}\prod_{j\neq i}(2k_j+1)\\
&=\sum_{i=1}^{n+1}\sum_{m=0}^n 2^m s_{m,n}(k_1,\dots,k_{i-1},k_{i+1},\dots,k_{n+1})\\
&=\sum_{m=0}^n 2^m \sum_{i=1}^{n+1} s_{m,n}(k_1,\dots,k_{i-1},k_{i+1},\dots,k_{n+1})\\
&=\sum_{m=0}^n 2^m (n+1-m)s_{m,n+1}.
\end{align*}
We thus have the following:
\begin{lem}\label{lem:Pdet}
Let $K=K(k_1,\dots,k_{2g+1}).$ Then
\[
\det(K)=\sum_{m=0}^{2g} 2^m (2g+1-m)s_{m,2g+1}.
\]
\end{lem}\qed

\begin{cor}\label{cor:Pvert}
Let $K=K(k_1,\dots,k_{2g+1}).$ Then
\[
\V(K)=g-\frac{3}{2}+\frac{1}{2}\sum_{m=0}^{2g} 2^m (2g+1-m)s_{m,2g+1}
\]
\end{cor}
\begin{proof}
Since $K$ is alternating, it follows from \cite{OSzAlt} that $\tau(K)=-\frac{1}{2}\sigma(K)=g.$  The desired result then follows from Lemmas \ref{lem:vert} and \ref{lem:Pdet}.
\end{proof}

We are now in a position to prove Theorem \ref{thm:pretz}.  We proceed by considering a few cases:
\begin{lem}\label{lem:pretz_2}
Let $K=K(k_1,\dots,k_{2+g1})$ with each $k_i\geq 0.$  If $g\geq 4$ and at least two $k_i$ are nonzero, then $K$ does not admit any chirally cosmetic surgeries along slopes of opposite signs.
\end{lem}
\begin{proof}
As remarked above, $\tau(K)=g(K)$ for these knots, and so by Corollary \ref{cor:combo}, if $K$ did admit chirally cosmetic surgeries of opposite signs, then
\[
v_3(K)(\V(K)+2g-1) = 7a_2(K)^2 - a_2(K)-10a_4(K).
\]
Notice by Lemma \ref{lem:av} that $a_2(K)>0$ and moreover, it follows from \cite{Crom} that because $K$ is a positive knot (i.e., one with all positive crossings), $a_4(K)\geq 0.$ Hence, we must have that:
\begin{equation}\label{eq:cont_2}
v_3(K)(\V(K)+2g-1) < 7a_2(K)^2.
\end{equation}
By Corollary \ref{cor:Pvert}, we see that:
\begin{align*}
\V(K)+2g-1 &= 3g-\frac{5}{2}+\frac{1}{2}\sum_{m=0}^{2g}2^m (2g+1-m)s_m \\
&\geq 3g-\frac{5}{2}+\frac{1}{2}(2g+1+ 4gs_1+4(2g-1)s_2)\\
&= 4g-2 + 2gs_1+2(2g-1)s_2.
\end{align*}
Here (and throughout the remainder of this section) we shall suppress the second subscript of the symmetric polynomial; so $s_i$ should be taken to mean $s_{i,2g+1}.$  Now applying Lemma \ref{lem:av} we find:
\begin{align*}
v_3(K)&(\V(K)+2g-1)\\
&\geq \frac{1}{3}g(g+1)(2g+1)(2g-1) 
+ \left(\frac{1}{3}g^2(g+1)(2g+1)+g(2g+1)(2g-1)\right)s_1 \\
 &\qquad+ (g^2(2g+1)+g(2g-1))s_1^2  
 + \left(\frac{1}{3}g(g+1)(2g+1)(2g-1)+2g(2g-1)\right)s_2\\
&\qquad+2g(2g-1)s_2^2
+ (2g^2+g(2g+1)(2g-1)+2g(2g-1))s_1s_2 + g^2s_1^3.
\end{align*}
As $g\geq 4$ by hypothesis, $2g-1\geq 7$ and $2g+1\geq 9$ so that:
\begin{align*}
v_3(K)&(\V(K)+2g-1)\\
 &\geq 
\frac{1}{3}g(g+1)(2g+1)(2g-1) + \left(3g^2(g+1)+g(2g+1)(2g-1)\right)s_1 \\
 &\qquad+ (9g^2+g(2g-1))s_1^2  + \left(\frac{1}{3}g(g+1)(2g+1)(2g-1)+2g(2g-1)\right)s_2\\
&\qquad+56s_2^2 + (2g^2+7g(2g+1)+14g)s_1s_2 + g^2s_1^3.
\end{align*}
Moreover, since at least 2 of the $k_i$ are nonzero, $s_1\geq 2$ and $s_2\geq 1.$ Hence,
\begin{align*}
v_3(K)&(\V(K)+2g-1)\\
& \geq \frac{1}{3}g(g+1)(2g+1)(2g-1) + \left(3g^2(g+1)+g(2g+1)(2g-1)\right)s_1\\
 &\qquad+ (9g^2+g(2g-1))s_1^2  + \frac{1}{3}g(g+1)(2g+1)(2g-1)+2g(2g-1) \\
&\qquad+56s_2^2 + 2g^2s_1+14g(2g+1)s_2+14gs_1s_2 + 4g^2s_1\\
&\geq \frac{2}{3}g(g+1)(2g+1)(2g-1) + \left(3g^2(g+1)+g(2g+1)(2g-1)+6g^2\right)s_1 \\
 &\qquad+ (9g^2+g(2g-1))s_1^2 +56s_2^2 +14g(2g+1)s_2+14gs_1s_2 \\
 &=\frac{2}{3}g(g+1)(4g^2-1) + (3g^3+3g^2+4g^3-g+6g^2)s_1 \\
 &\qquad+ (9g^2+g(2g-1))s_1^2 +56s_2^2 +14g(2g+1)s_2+14gs_1s_2\\
 &\geq \frac{2}{3}g(g+1)(4g^2-1) + (7g^3+7g^2)s_1 + 9g^2s_1^2. +56s_2^2 +14g(2g+1)s_2 +14gs_1s_2
\end{align*}
By Lemma \ref{lem:av}, we see that:
\begin{align*}
v_3(K)&(\V(K)+2g-1)-7a_2^2(K)\\
& \geq \frac{2}{3}g(g+1)(4g^2-1) + (7g^3+7g^2)s_1 + 9g^2s_1^2. +56s_2^2 +14g(2g+1)s_2 +14gs_1s_2\\
&\qquad - \left(\frac{7}{4}g^2(g+1)^2 + 7g^2(g+1)s_1+7g^2s_1^2+7g(g+1)s_2+7s_2^2+14gs_1s_2\right)\\
&=g(g+1)\left(\frac{2}{3}(4g^2-1)-\frac{7}{4}g(g+1)\right) + 2g^2s_1^2 + 47s_2^2 + 7g(3g+1)s_2\\
&\geq g(g+1)\left(\frac{11}{12}g^2-\frac{7}{4}g-\frac{2}{3}\right)\\
&\geq g(g+1)\left(\frac{23}{12}g-\frac{2}{3}\right)\\
&>0,
\end{align*}
where we have used the hypothesis that $g\geq 4$ in the last two inequalities.  Thus we have found a contradiction to \eqref{eq:cont_2}, as desired.
\end{proof}
Now we turn our attention to the case where at most one $k_i$ is nonzero; in fact, we may assume it is $k_1.$  One sees that these knots are exactly the ``odd" double-twist knots of the form $J(-(2k+1),2g)$ (see Figure \ref{fig:Dtwist} for an illustration). 
\begin{figure}
\centering
\includegraphics[width=.75\textwidth]{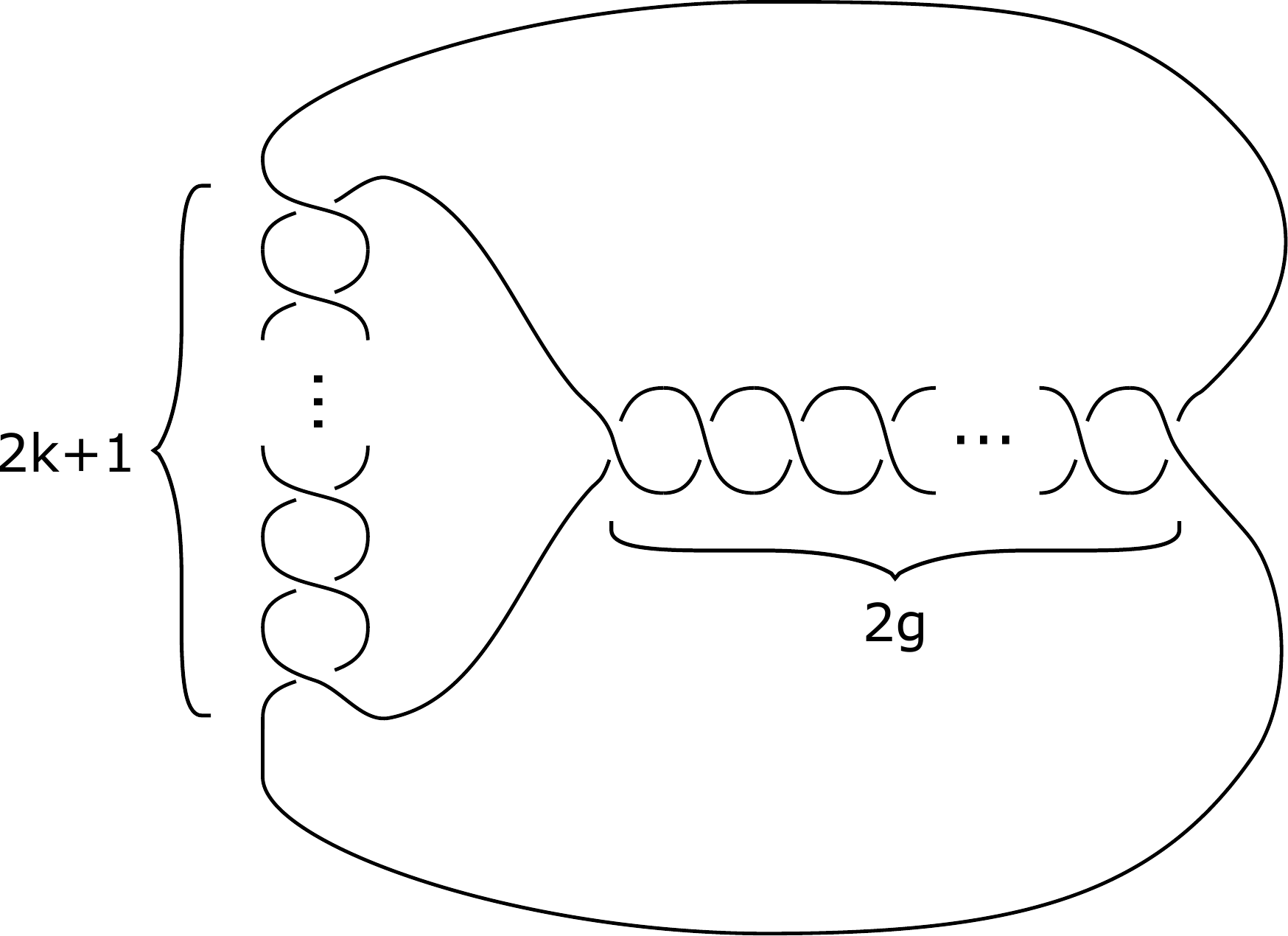}
\caption{The odd double twist knot $J(-(2k+1),2g).$  Here the twisting regions correspond to $k,g\geq 0.$}
\label{fig:Dtwist}
\end{figure}
\begin{lem}\label{lem:Ja4}
Let $K=J(-(2k+1),2g)$ for $k,g\geq 0.$ Then
\[
a_4(K)=\binom{g+2}{4}+k\binom{g+1}{3} = \frac{(g+2)(g+1)g(g-1)}{24}+\frac{(g+1)g(g-1)}{6}k.
\]
\end{lem}
\begin{proof}
Applying the skein relation for the Conway polynomial \eqref{eq:conwaySkein} to a crossing in the odd twisting region, we find that
\[
\nabla_{J(-(2i+1),2g)}(z)-\nabla_{J(-(2i-1),2g)}(z) = z\nabla_{T(2,2g)}(z),
\]
where $T(2,2g)$ is the $(2,2g)$-torus link (with positive linking number).  Summing equations of the above form from $i=1$ to $k$ (and observing that $J(-1,2g)=T(2,2g+1)$) gives:
\begin{equation}\label{eq:J_skein}
\nabla_{J(-(2k+1),2g)}(z)= \nabla_{T(2,2g+1)}(z)+kz\nabla_{T(2,2g)}(z).
\end{equation}
The Conway polynomials of $(2,n)$-torus links are well-known (and can be derived inductively by repeatedly applying the skein relation \eqref{eq:conwaySkein}):
\begin{align*}
\nabla_{T(2,2g)}(z)&=\binom{g}{1}z+\binom{g+1}{3}z^3+\dots\\
\nabla_{T(2,2g+1)}(z)&=1+\binom{g+1}{2}z^2+\binom{g+2}{4}z^4+\dots
\end{align*}
Substituting these back into \eqref{eq:J_skein} gives:
\[
\nabla_{J(-(2k+1),2g)}(z) = 1 + \left[\binom{g+1}{2}+kg\right]z^2+\left[\binom{g+2}{4}+k\binom{g+1}{3}\right]z^4+\dots
\]
We now simply read off $a_4(J(-(2k+1),2g))$ as the coefficient of $z^4.$
\end{proof}
Recalling that $J(-(2k+1),2g) = K(k,0,\dots,0),$ we compute from Lemma \ref{lem:av} and Corollary \ref{cor:Pvert}:
\begin{equation}\label{eq:J_av}
\begin{cases}
a_2(J(-(2k+1),2g)) = \frac{1}{2}g(g+1)+kg\\
v_3(J(-(2k+1),2g)) = \frac{1}{2}\left(\frac{g(g+1)(2g+1)}{3}+g(2g+1)k+gk^2\right)\\
\V(J(-(2k+1),2g)) = 2g-1 + 2gk
\end{cases}
\end{equation}
\begin{lem}\label{lem:J_5}
Let $K=J(-(2k+1),2g)).$  If $k>0$ and $g\geq 5$ then $K$ admits no chirally cosmetic surgeries along slopes of opposite signs.
\end{lem}
\begin{proof}
As before, we suppose for the sake of a contradiction that $K$ admits chirally cosmetic surgeries along slopes of opposite signs.  Then by Corollary \ref{cor:combo}, we must have
\[
v_3(K)(\V(K)+2g-1) = 7a_2(K)^2 - a_2(K)-10a_4(K).
\]
Hence, from Lemma \ref{lem:Ja4} and \eqref{eq:J_av}, we compute:
\begin{align*}
0 &= v_3(K)(\V(K)+2g-1) - 7a_2(K)^2 + a_2(K)+10a_4(K)\\
&= \frac{1}{2}\left(\frac{g(g+1)(2g+1)}{3}+g(2g+1)k+gk^2\right)(4g-2+2gk) - 7\left(\frac{1}{2}g(g+1)+kg\right)^2\\
&\qquad  + \frac{1}{2}g(g+1)+kg + \frac{5(g+2)(g+1)g(g-1)}{12}+\frac{5(g+1)g(g-1)}{3}k\\
& = \frac{1}{3}g(g+1)(2g+1)(2g-1)-\frac{7}{4}g^2(g+1)^2+\frac{1}{2}g(g+1) + \frac{5(g+2)(g+1)g(g-1)}{12} \\
&\qquad + \left(g(2g+1)(2g-1)+\frac{1}{3}g^2(g+1)(2g+1)-7g^2(g+1) + g + \frac{5(g+1)g(g-1)}{3}\right)k\\
&\qquad +\left(g(2g-1)+g^2(2g+1)-7g^2\right)k^2 + g^2k^3\\
&= g(g+1)\left(-\frac{4}{3}g-\frac{2}{3}\right) + g\left(\frac{2}{3}g^3-\frac{1}{3}g^2-\frac{20}{3}g-\frac{5}{3}\right)k + g(2g^2-4g-1)k^2+g^2k^3.
\end{align*}
Since $g\geq 5$ by hypothesis, $\frac{2}{3}g^3-\frac{1}{3}g^2-\frac{20}{3}g-\frac{5}{3}\geq 3g^2-\frac{20}{3}g-\frac{5}{3}\geq \frac{25}{3}g-\frac{5}{3}>0.$ Similarly, $2g^2-4g-1\geq 6g-1>0.$  Hence, since $k\geq 1,$ we have that:
\begin{align*}
0 &\geq g(g+1)\left(-\frac{4}{3}g-\frac{2}{3}\right) + g\left(\frac{2}{3}g^3-\frac{1}{3}g^2-\frac{20}{3}g-\frac{5}{3}\right) + g(2g^2-4g-1)+g^2\\
&=g\left(\frac{2}{3}g^3+\frac{1}{3}g^2-\frac{35}{3}g-\frac{10}{3}\right)\\
&\geq g\left(\frac{20}{3}g-\frac{10}{3}\right)\\
&>0,
\end{align*}
where again the last two inequalities used the assumption that $g\geq 5.$  The contradiction completes the proof.
\end{proof}
Finally, we consider the case of $g=4;$ that is, $K=J(-(2k+1),8).$  From the proof of Lemma \ref{lem:J_5}, we compute that:
\begin{align*}
v_3(K)&(\V(K)+2g-1) - 7a_2(K)^2 + a_2(K)+10a_4(K)\\
&= \left. g(g+1)\left(-\frac{4}{3}g-\frac{2}{3}\right) + g\left(\frac{2}{3}g^3-\frac{1}{3}g^2-\frac{20}{3}g-\frac{5}{3}\right)k + g(2g^2-4g-1)k^2+g^2k^3\right|_{g=4}\\
&= -120 + 36k + 60k^2 + 16k^3.
\end{align*}
It is clear that if $k\geq 2$ then this quantity is positive.  On the other hand, when $k=1,$ this quantity is -8.  Once again, we may conclude by Corollary \ref{cor:combo}:
\begin{lem}\label{lem:J_4}
Let $K=J(-(2k+1),8)).$  If $k>0$ then $K$ admits no chirally cosmetic surgeries along slopes of opposite signs.
\end{lem}\qed

We now complete the classification of chirally cosmetic surgeries for the alternating odd pretzel knots.
\begin{proof}[Proof of Theorem \ref{thm:pretz}]
Theorem 6.4 of \cite{IIS} covers the genus 1 case, while Theorem 1.1 of \cite{CCP} deals with the cases $g=2,3.$  Now, we see that Lemmas \ref{lem:pretz_2}, \ref{lem:J_5}, and \ref{lem:J_4} rule out chirally cosmetic surgeries along slopes of opposite signs for $g \geq 4$ and at least one $k_i>0.$  Hence for these knots, we only need to consider slopes of the same sign.  However, Theorem \ref{thm:rk} implies that any knot admitting such cosmetic surgeries must be an L-space knot, but $K=K(k_1,\dots, k_{2g+1})$ is not an L-space knot whenever at least one $k_i>0.$  This can be seen, for instance, by noting that the immersed curve invariant of such a knot  will have at least one figure-eight component, as in the discussion leading to the proof of Lemma \ref{lem:vert}.  Alternatively, one can see that the Alexander polynomial does not have the required form (e.g., the leading coefficient is not $\pm 1$ so that $K$ is not even fibered).
\end{proof}
%\section*{Acknowledgements}
%The author would like to thank Professor Zolt\'{a}n Szab\'{o}  for suggesting and encouraging work on this problem.  This work was supported by the NSF RTG grant DMS-1502424.

\subsection{Whitehead doubles}
We now turn our attention to Whitehead doubles.  Let us first collect some facts: the Conway polynomial of $D_+(K,n)$ is well-known to be $1-nz^2,$ so that $a_2(D_+(K,n)) = -n$ nad $a_4(D_+(K,n))=0.$  Moreover, by Proposition 5.1 of \cite{IW}, we also know that $v_3(D_+(K,n)) = -2a_2(K) + \frac{n^2-n}{2}.$  Whitehead doubles have genus (at most) one, and we have the following due to Hedden (Theorem 1.5 of \cite{Hed}):
\begin{equation}\label{eq:WD_tau}
\tau(D_+(K,n)) = \begin{cases}
1 & n<2\tau(K) \\
0 & n\geq 2\tau(K).
\end{cases}
\end{equation}

Our first result concerns rules out cosmetic surgeries for zero-framed Whitehead doubles in a broad range of cases; it follows purely from finite-type obstructions:
\begin{thm}
Given a knot $K$ with $a_2(K)\neq 0,$  then $D_+(K,0)$ does not admit any chirally cosmetic surgeries.
\end{thm}
\begin{proof}
Suppose $D_+(K,0)$ admitted chirally cosmetic surgeries along slopes $\frac{p}{q}$ and $\frac{p}{q'}.$ By the above formulas, we see that in this case, $a_2(D_+(K,0)) = a_4(D_+(K,0)) = 0,$ while $v_3(D_+(K,0))\neq 0.$  Then Theorem \ref{thm:ft} implies that $p=0,$ but then the two slopes were not distinct to begin with.
\end{proof}

For our next result, we make use of an obstruction due to Ichihara, Ito, and Saito obtained by combining Theorem \ref{thm:ft} with some Casson-type invariant obstructions.
\begin{lem}[Theorem 6.1 of \cite{IIS}]\label{lem:ft2}
If a knot $K$ admits chirally cosmetic surgeries along slopes $\frac{p}{q}$ and $\frac{p}{q'}$ with $q+q'\neq 0,$ then:
\[
4|a_2(K)| \leq d(K)\left|\frac{7a_2(K)^2-a_2(K)-10a_4(K)}{2v_3(K)}\right|,
\]
where $d(K)$ is the degree of the Conway polynomial of $K.$
\end{lem}

\WD*
\begin{proof}
We note that in this case, $v_3(D_+(K,n))=\frac{1}{2}(n^2-n),$ and so is nonzero when $|n|\geq 5.$  Thus, Theorem \ref{thm:ft} implies that $D_+(K,n)$ does not admit any chirally cosmetic surgeries along opposite slopes.  Hence, we may apply Lemma \ref{lem:ft2} to see that if $D_+(K,n)$ admits chirally cosmetic surgeries, then:
\begin{align*}
4|n| &\leq 2\left|\frac{7n^2+n}{n^2-n}\right|\\
&=2\left|\frac{7n+1}{n-1}\right|\\
&=2\left|7+\frac{8}{n-1}\right|.
\end{align*}
Since $|n|\geq 5,$ we find that
\[
10\leq 2|n| \leq \left|7+\frac{8}{n-1}\right| \leq 9,
\]
a contradiction.  Now we turn to the second claim: notice that the hypothesis that $2\tau(K)>n$ is exactly what is required in order for $\tau(D_+(K,n))$ to equal 1, according to \eqref{eq:WD_tau}.  Since $D_+(K,n)$ is genus 1, we may apply Corollary \ref{cor:combo} to see that if $D_+(K,n)$ admits chirally cosmetic surgeries along slopes of opposite signs, then:
\[
\frac{n^2-n}{2}(\V(D_+(K,n))+1) = 7n^2+n.
\]
If $n=1,$ this leads to a contradiction.  Otherwise, this implies:
\[
\frac{\V(D_+(K,n))+1}{2} = 7+\frac{8}{n-1}.
\]
Since $\tau(D_+(K,n))=1,$ one can see that $\V(D_+(K,n))$ must be odd, and so $ 7+\frac{8}{n-1}$ must be an integer, which is impossible for $n=\pm 4,-2.$
\end{proof}
\begin{rmk}
It is possible to prove similar finiteness results for different values of $a_2(K)$ using the same methods (see also Theorem 6.3 of \cite{IIS}).  Then, there will similarly be certain values of $n$ which can be excluded under extra hypotheses of the type $2\tau(K)>n.$
\end{rmk}

\section{L-space knots}
In this section, we prove:
\lspace*
Recall that a closed 3-manifold $M$ is said to be an \textit{L-space} if ${\rk\widehat{HF}(M,\mathfrak{s})=1}$ for each $\mathfrak{s}\in Spin^c(M).$  A knot $K$ is called an \textit{L-space knot} if $S^3_K(r)$ is an L-space for some $r\in \mathbf{Q}.$ We shall make use of the following criterion for a knot to be an L-space knot in the immersed curve context:
\begin{lem}\label{lem:lspace}
Given a knot $K,$ let $\overline{\Gamma} \subset \overline{T}_{\bullet}$ be the (standard lift of) the immersed curve invariant of $K.$  If $K$ is an L-space knot, then $\overline{\Gamma}$ consists of only one component ($\gamma_0$), which can by made embedded by a homotopy.
\end{lem}
\begin{proof}
This follows from standard results for the Knot Floer complex of an L-space knot, but it can be seen more directly. Let $K$ be a knot such that $S^3_K(\frac{p}{q})$ is an L-space.  Denoting the immersed curve invariant of $K$ by $\Gamma$ and its standard lift to $\overline{T}_{\bullet}$ by $\overline{\Gamma}.$  For any integer $i,$ there is a lift $\overline{L}_i \subset \overline{T}_{\bullet}$ of $L_{p/q}$ that intersects $\mu$ at level $i,$ and hence intersects $\overline{\Gamma}$ at least $n_i(K)$ times. By Theorem \ref{thm:pairS}, this implies that there exists a Spin$^c$-structure $\mathfrak{s}$ such that $\rk \widehat{HF}(S^3_K(\frac{p}{q}),\mathfrak{s}) \geq n_i(K).$ By the hypothesis that $S^3_K(\frac{p}{q})$ is an L-space, this implies that $n_i(K)\leq 1.$

Now, suppose $\overline{\Gamma}$ has a component that is not $\gamma_0.$ By the structure of $\overline{\Gamma},$ it must be a closed curve supported in a neighborhood of the midline $\mu,$ (and in fact, it cannot simply wrap around a single puncture), so it must ``double-back" on itself at some point.  This would contribute at least 2 to some $n_i(K),$ which is impossible by the above.  For the same reason, $\gamma_0$ must also not ``double-back" on itself, and so it must be embedded when ``pulled tight."
\end{proof}

\begin{figure}
\centering
\includegraphics[height = .4\textheight]{alt_immersed_1}
\hspace{.2\textwidth}
\includegraphics[height = .4\textheight]{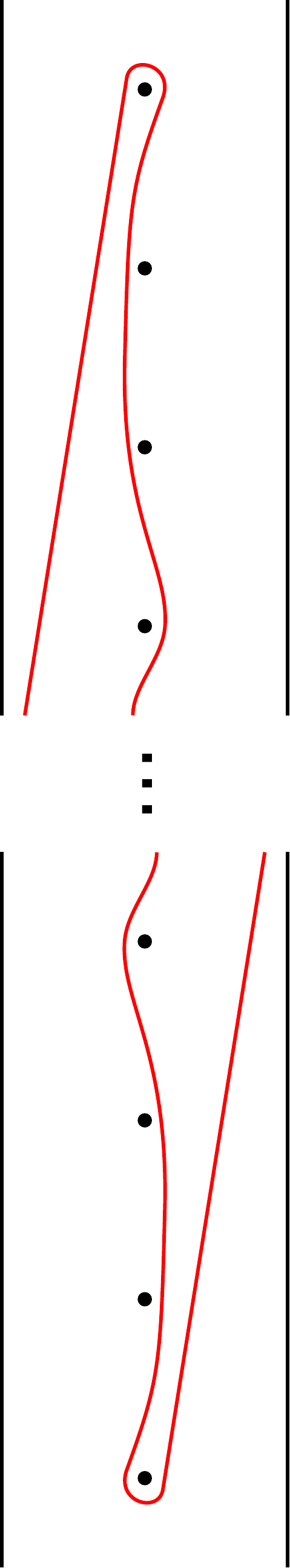}
\hspace{.2\textwidth}
\includegraphics[height = .4\textheight]{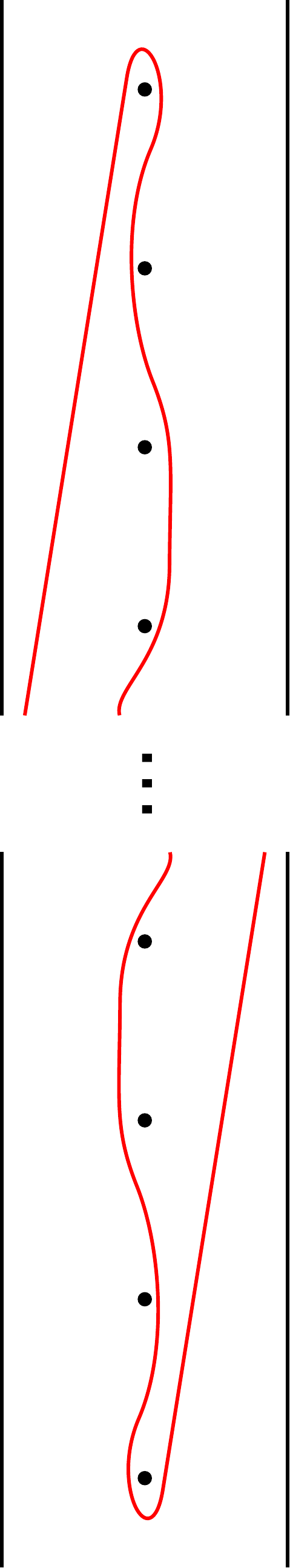}

\caption{Examples of different kinds of immersed curve invariants for L-space knots.}
\label{fig:lspace_ex}
\end{figure}

Hence, the immersed curve invariant for L-space knots consists of a single curve which ``weaves" between the punctures at various heights; see Figure \ref{fig:lspace_ex} for some examples of this behavior.  From this structure result, we can immediately see the following:
\begin{cor}\label{cor:LS}
Let $K$ be an L-space knot. Then $|\tau(K)|=g(K)$ and $\V(K)=2g(K)-1.$
\end{cor}
Up to mirror image, we may suppose without loss of generality that $\tau(K)>0$ for an L-space knot $K.$ We shall make this assumption for the rest of this section.
\subsection{Gradings}
Here we discuss how to obtain relative grading information from the pairing of immersed curve invariants. Recall that, for any knot $K$ and $\frac{p}{q} \neq 0,$ the vector space $\widehat{HF}(S^3_K(\frac{p}{q}))$ has a relative $\mathbf{Z}$-grading called the \textit{Maslov} grading, denoted $M(x)$ for any generator $x$ (the case of zero-framed surgery will not be of concern to us, as it cannot be part of a cosmetic surgery pair). The relative grading between two generators may be computed from the immersed curves as follows: suppose $x,y \in L_{p/q} \cap \Gamma$ and suppose further that there exists an immersed bigon in $T_{\bullet}$ with one edge lying in $L_{p/q}$ and the other in $\Gamma$ as in Figure \ref{fig:bigon}. Let $k$ be the signed count of the number of times the bigon crosses the puncture/basepoint.  Then:
\begin{equation}\label{eq:grading}
M(x)-M(y)=1-2k.
\end{equation}
The same computation may be performed in the covering space $\overline{T}_{\bullet},$ and moreover, it also still holds in the further covering space $\widetilde{T}_{\bullet} = \mathbf{R}^2\setminus \mathbf{Z}\times (\frac{1}{2}+\mathbf{Z}),$ corresponding to the universal covering space of the torus.
\begin{figure}
\centering
\includegraphics[width=.25\textwidth]{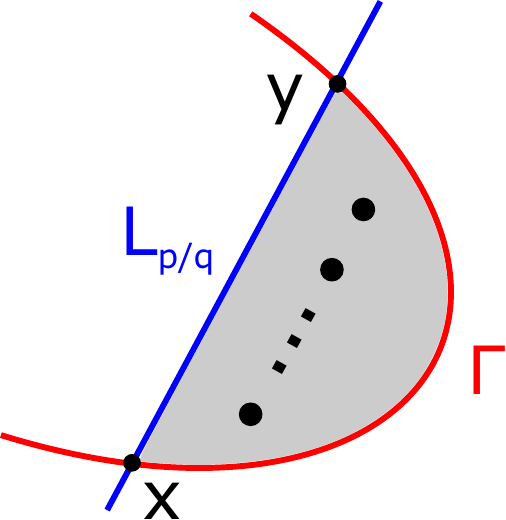}
\caption{A bigon from $x$ to $y.$}
\label{fig:bigon}
\end{figure}
\begin{rmk}
In general, for a given pair of generators, there need not exist such a bigon between them (for instance, if the generators are in distinct Spin$^c$-structures or if they belong to distinct components of $\Gamma$). However, we are only concerned with L-space knots in this section, and all relative grading information within any given Spin$^c$-structure may be deduced from equation \eqref{eq:grading} for surgeries on those knots.
\end{rmk}

For surgery on an L-space knot, the relative grading structure within each Spin$^c$-structure has a particularly simple structure:  the gradings of the generators may be arranged in a ``zig-zag" pattern. More precisely, define
\[
\mathcal{Z}_n := \left.\left\lbrace (m_i)_{i=1}^{2n+1}\in \mathbf{Z}^{2n+1} \left| \begin{matrix}
m_i-m_{i+1} \in 2\mathbf{Z}_{>0} - 1 &\textrm{ for }i \textrm{ even}\\
m_{i+1}-m_i \in 2\mathbf{Z}_{>0} - 1 &\textrm{ for }i \textrm{ odd}
\end{matrix} \right. \right\rbrace\right/ \sim
\]
where $(m_i) \sim (m'_i)$ if there exists some $N\in \mathbf{Z}$ and  $\sigma \in \mathfrak{S}_{2n+1}$ such that $m'_i = N + m_{\sigma(i)}$ for each $i.$ Put $\mathcal{Z} := \bigcup_{n\geq 0} \mathcal{Z}_n.$

\begin{lem}
Let $K$ be an L-space knot and let $\frac{p}{q}>0$ be given.  Then for each $\mathfrak{s}\in Spin^c(S^3_K(\frac{p}{q})),$ the gradings of the generators of $\widehat{HF}(S^3_K(\frac{p}{q}),\mathfrak{s})$ may be arranged into a sequence $(x_i)\in \mathcal{Z}.$
\end{lem}
\begin{proof}
Let $\Gamma$ be the immersed curve invariant associated to $K.$ We consider the lift $\widetilde{\Gamma}\subset \widetilde{T}_{\bullet}$ and a lift $\widetilde{L}_{p/q}$ of $L_{p/q}$ to $\widetilde{T}_{\bullet}$ corresponding to $\mathfrak{s};$ note that this latter lift is not unique, although it is unique up to certain horizontal translations. One may take $\widetilde{L}_{p/q}$  to be a line of slope $\frac{p}{q},$ and with $\widetilde{\Gamma}$ in ``pulled tight" configuration, the generators of $\widehat{HF}(S^3_K(\frac{p}{q})$ correspond to the points in $\widetilde{\Gamma} \cap \widetilde{L}_{p/q}.$ By minimality, we see that any bigon between any two such points must cover at least one puncture. Label the intersection points $x_1,x_2,\dots$ from left to right, as in Figure \ref{fig:LSbigons}. We see that adjacent intersection points are connected by a bigon which covers some punctures, and moreover, such bigons alternate begin above or below the line $\widetilde{L}_{p/q}.$ As $\frac{p}{q}>0$ by hypothesis, we see that the leftmost bigon must be below the line, while the rightmost must be above it. Hence, there are an odd number of generators $x_1,\dots,x_{2n+1},$ and for each even $2\leq i \leq 2n,$ there is a bigon from $x_{i-1}$ to $x_i$ as well as one from $x_{i+1}$ to $x_i.$ Applying formula \eqref{eq:grading}, this implies that $M(x_i)-M(x_{i-1}) = 2k-1$ for some $k \in \mathbf{Z}_{>0}$ (corresponding to the number of punctures in the bigon), and similarly $M(x_i)-M(x_{i+1}) = 2k'-1$ for some $k' \in \mathbf{Z}_{>0}.$  Hence, the equivalence class of $(x_1,\dots,x_{2n+1})$ is in $\mathcal{Z}_n.$
\end{proof}
\begin{figure}
\centering
\includegraphics[width=.65\textwidth]{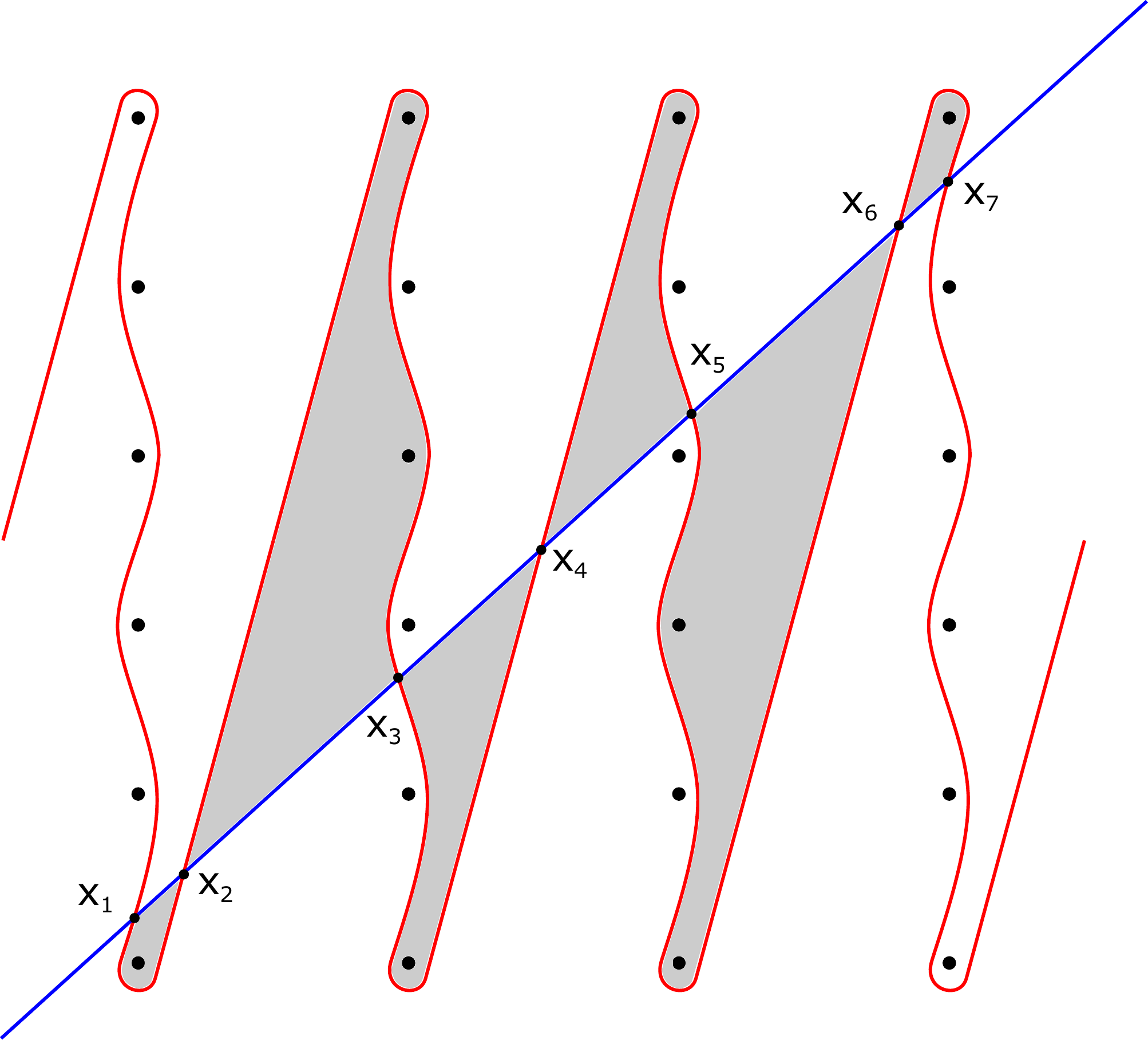}
\caption{An illustrative example of relative grading computation for sugery on an L-space knot.}
\label{fig:LSbigons}
\end{figure}

\subsection{Proof of Theorem \ref{thm:lspace}}
We are now in a position to prove Theorem \ref{thm:lspace}. First, we give a geometric interpretation to the constraint given in Theorem \ref{thm:main}. Recall from Lemma \ref{lem:lspace} that the immersed curve invariant of an L-space knot consists of a single embedded component. In the covering space $\widetilde{T}_{\bullet},$ this curve may be seen, after an appropriate homotopy, to consist entirely of (slight perturbations of) vertical segments between heights $-2g+\frac{1}{2}$ and $2g-\frac{1}{2}$ where $g=g(K)$ is the genus of the knot, as well as diagonal segments of slope $2g-1.$ Let us now introduce punctures along the diagonal segments, at the same heights as the existing punctures (that is, at half-integer values in $(-2g,2g)$ and we perturb the diagonal segments to ``weave" in between the new punctures in a symmetric way as the vertical segments weave through the original punctures. See Figure \ref{fig:LSpunc} for an illustration; we depict the new punctures with an asterisk ($\ast$), and we denote the new space by $\widetilde{T}_{\bullet\ast}.$
\begin{figure}
\centering
\includegraphics[scale=.18]{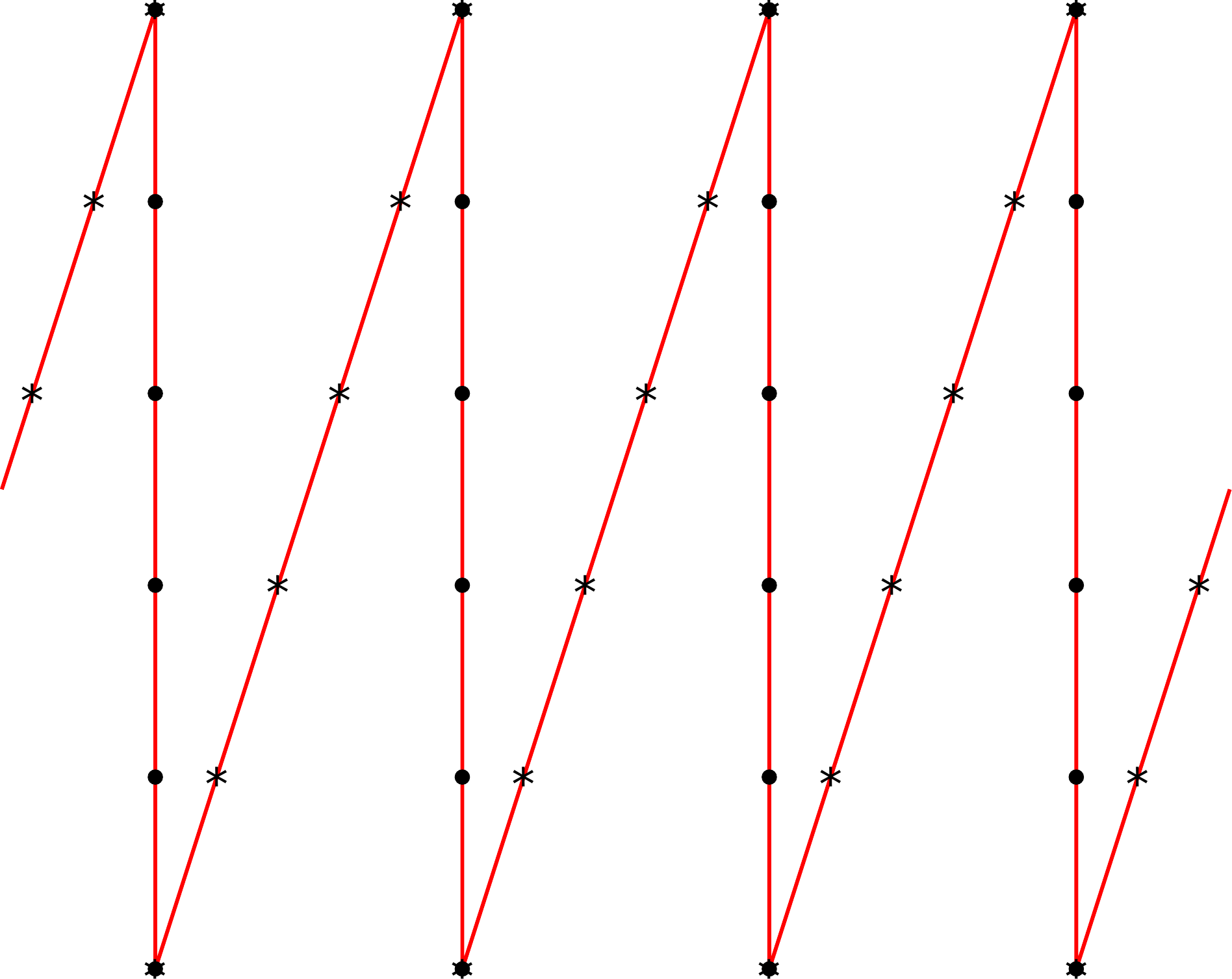} \hfill
\includegraphics[scale=.18]{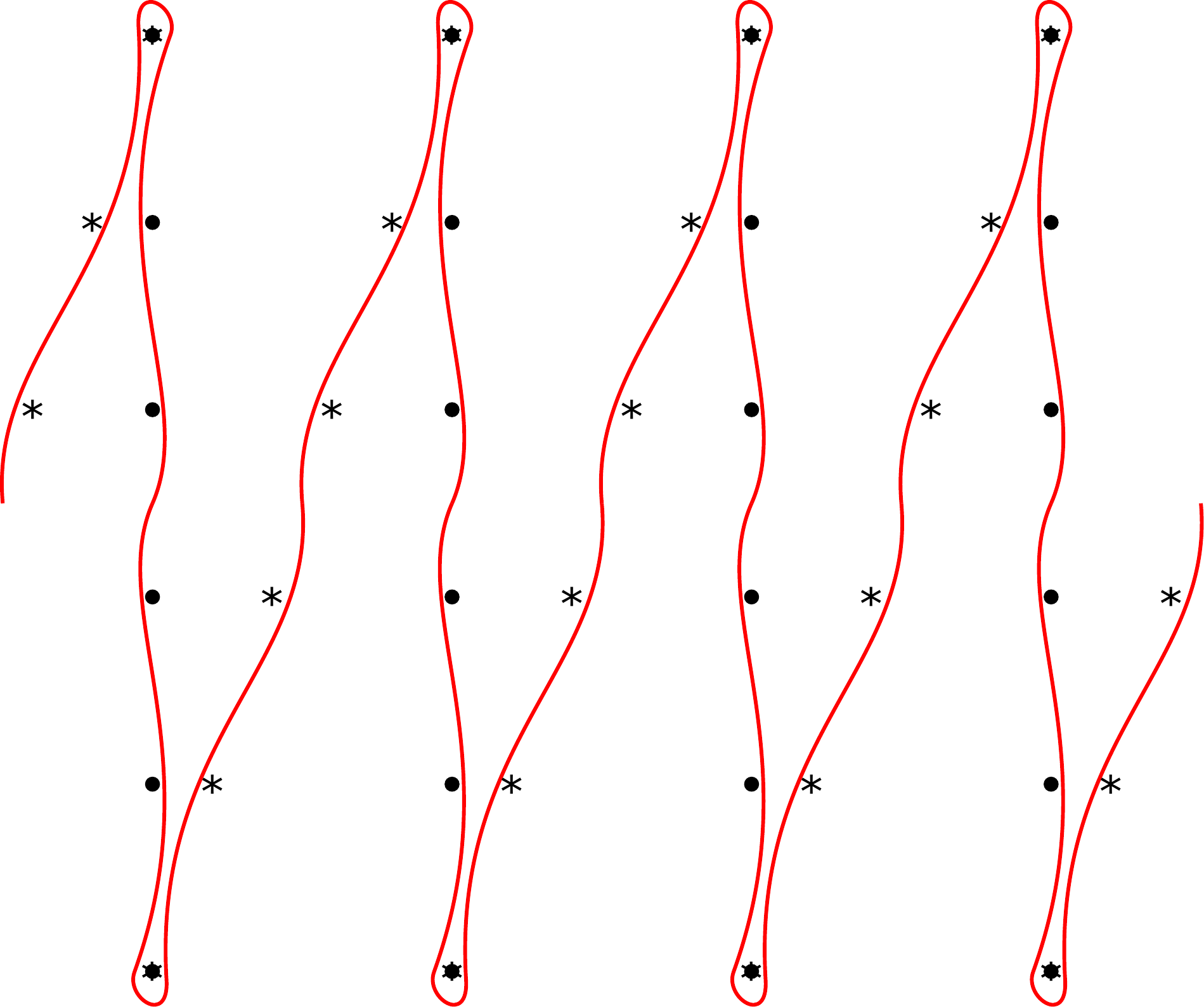}
\caption{On the left, a schematic of the general structure of the immersed curve invariant of an L-space knot lifted to $\widetilde{T}_{\bullet\ast}.$ On the right, the curve after a perturbation to avoid both kinds of punctures; the perturbations are done in a symmetric way with respect to the two kinds of punctures.}
\label{fig:LSpunc}
\end{figure}

Consider the shear transformation $\Phi_g:\mathbf{R}^2\rightarrow\mathbf{R}^2$ given by the matrix $\begin{pmatrix}
1 & -\frac{1}{2g-1}\\
0&1
\end{pmatrix}.$ Under this transformation, we see that the $\ast$ punctures in $\widetilde{T}_{\bullet\ast}$ become vertically arranged, and the $\bullet$ punctures become diagonally arranged. Moreover, $\Phi_g(\widetilde{\Gamma})$ is a mirror image of $\widetilde{\Gamma}$ so that, after performing the vertical reflection $r_V:(x,y)\mapsto(x,-y)$, we see that $r_V(\Phi_g(\widetilde{\Gamma}))=\widetilde{\Gamma}$ setwise. Under $r_V \circ \Phi_g,$ the two kinds of punctures become interchanged; see Figure \ref{fig:LStransform} for an illustration.
\begin{figure}
\centering
\begin{subfigure}{.4\textwidth}
\centering
\includegraphics[scale=.15]{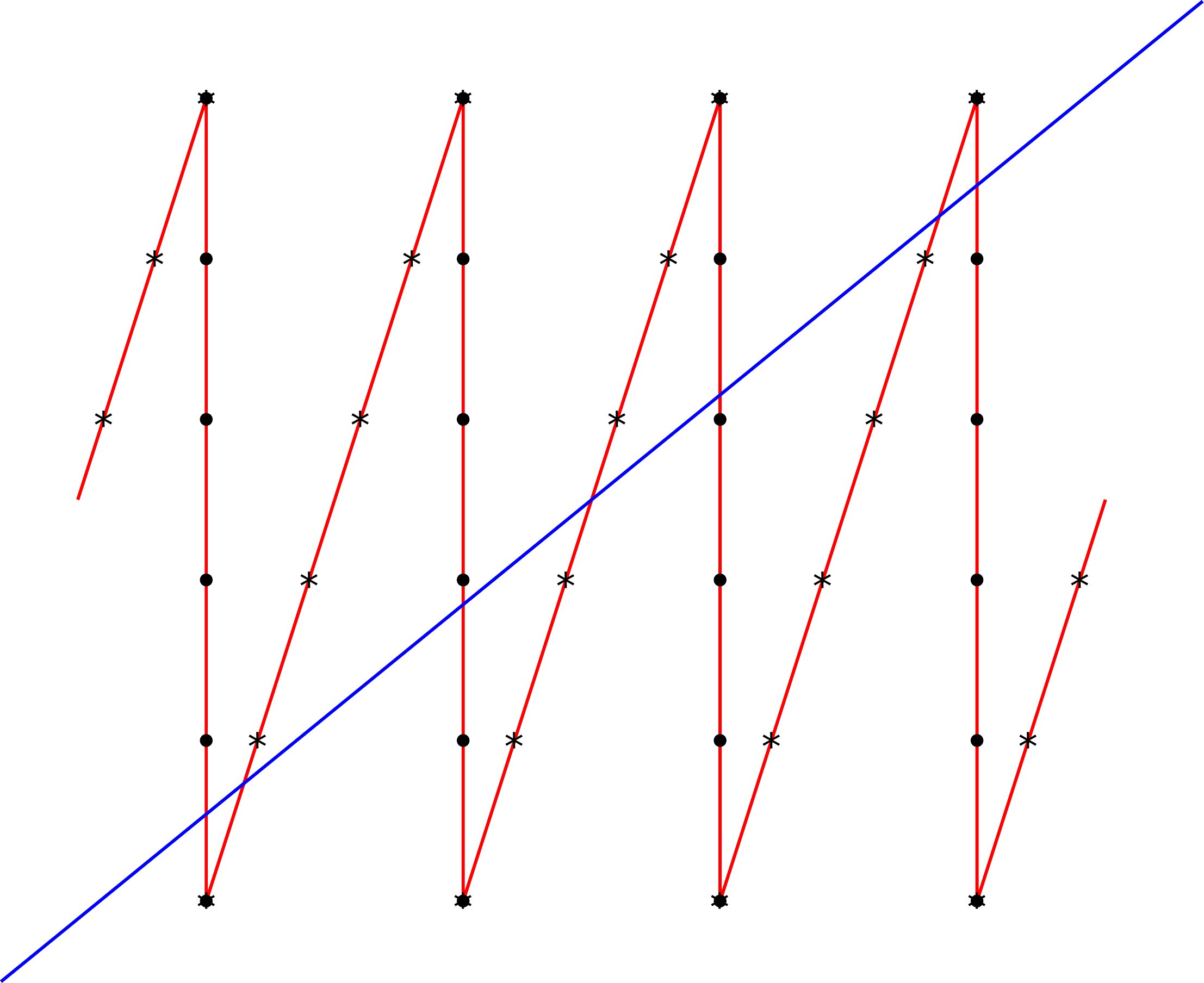}
\caption{Intersection of an immersed curve invariant $\widetilde{\Gamma}$ with the line $\widetilde{L}_{p/q}.$}
\label{fig:start}
\end{subfigure}\hfill
\begin{subfigure}{.4\textwidth}
\centering
\includegraphics[scale=.15]{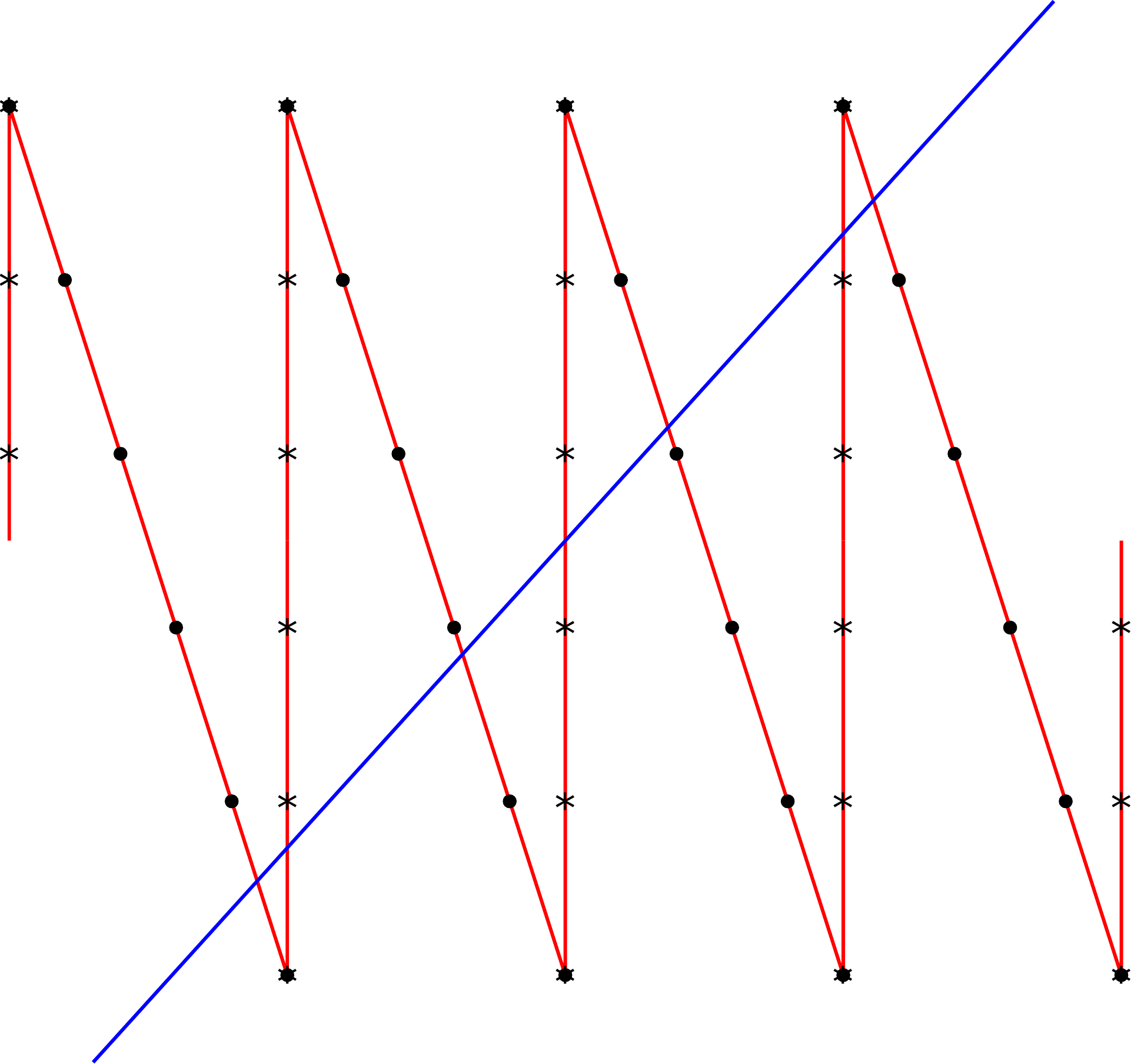}
\caption{After performing the shear transformation $\Phi_g.$}
\label{fig:shear}
\end{subfigure}
\\
\begin{subfigure}{.4\textwidth}
\centering
\includegraphics[scale=.15]{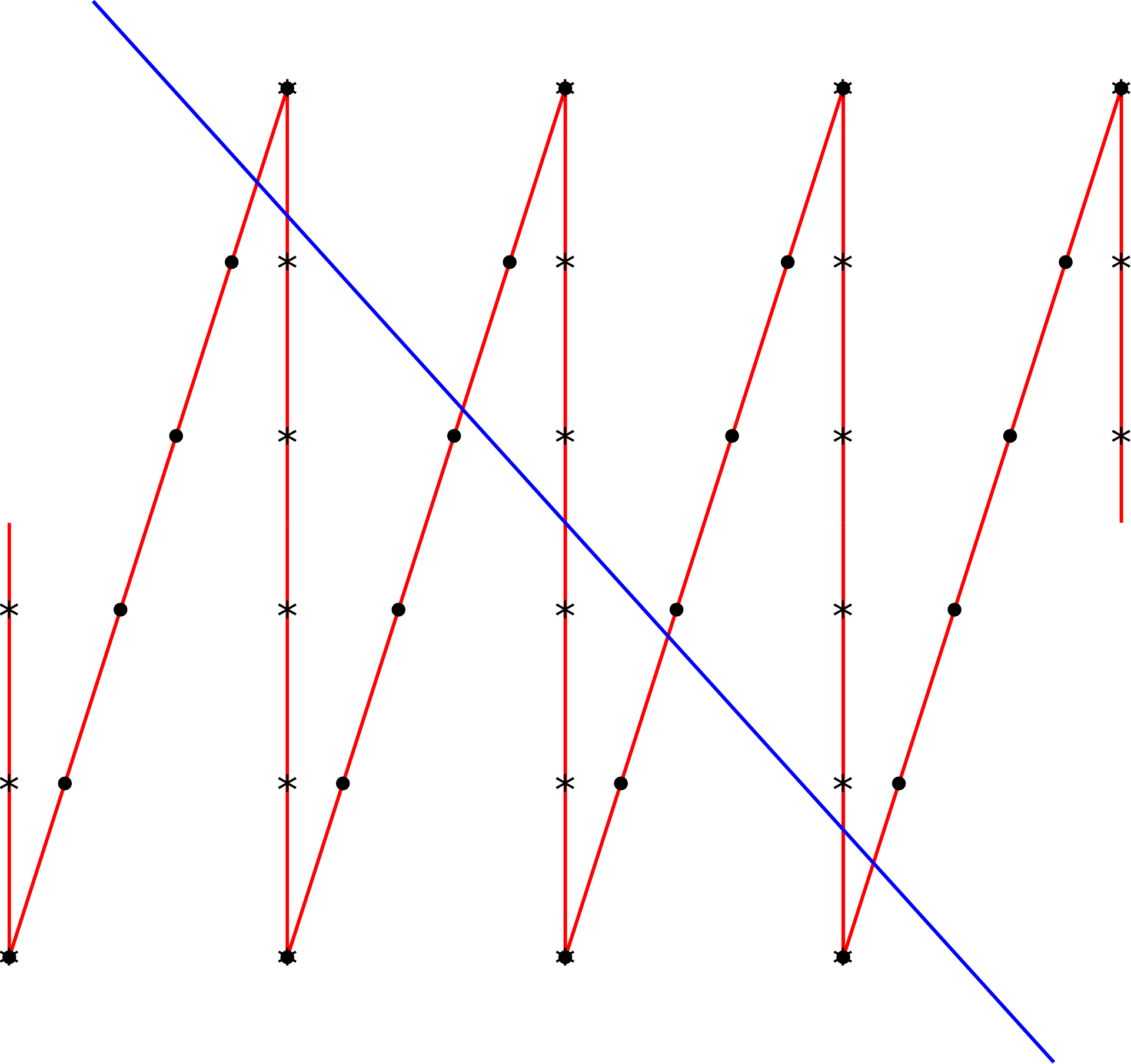}
\caption{After a vertical reflection.}
\label{fig:ref}
\end{subfigure}
\caption{The sequence of transformations that give rise to the bijection $psi$ of Spin$^c$-structures.}
\label{fig:LStransform}
\end{figure}
Now, given a slope $\frac{p}{q}>0$ and an $\mathfrak{s}\in Spin^c(S^3_K(\frac{p}{q})),$ there exists a lift $\widetilde{L^{\mathfrak{s}}}_{p/q}$ of $L_{p/q}$ such that the intersection with $\widetilde{\Gamma}$ corresponds to $\widehat{HF}(S^3_K(\frac{p}{q}),\mathfrak{s})$ as in Theorem \ref{thm:pairS} (forgetting about the $\ast$ punctures). This is a line of slope $\frac{p}{q},$ and so $r_V(\Phi_g(\widetilde{L^{\mathfrak{s}}}_{p/q}))$ is a line of slope $\frac{p}{-q+p/(2g-1)}.$ Thus, its intersection with $\widetilde{\Gamma}$ corresponds to $\widehat{HF}(S^3_K(\frac{p}{-q+p/(2g-1)}),\mathfrak{t})$ for some $\mathfrak{t}\in Spin^c(S^3_K(\frac{p}{-q+p/(2g-1)}))$ (forgetting about the $\bullet$ punctures). Hence, we see that $r_V\circ \Phi_g$ induces a bijection $\psi:Spin^c(S^3_K(\frac{p}{q})) \rightarrow Spin^c(S^3_K(\frac{p}{-q+p/(2g-1)}))$ such that $\rk  \widehat{HF}(S^3_K(\frac{p}{q}),\mathfrak{s}) = \rk \widehat{HF}(S^3_K(\frac{p}{-q+p/(2g-1)}),\psi(\mathfrak{s}))$ for each $\mathfrak{s}\in Spin^c(S^3_K(\frac{p}{q})).$ This map is well-defined, as a lift $\widetilde{L}$ corresponding to a given Spin$^c$-structure is unique up to horizontal integral translations.
Note that $\frac{p}{-q+p/(2g-1)}<0$ exactly when $\frac{p}{q}<2g-1,$ and in fact, if a genus $g$ L-space knot admits chirally cosmetic surgeries along slopes of opposite signs, the slopes must be related by the above construction:
\begin{lem}\label{lem:LSslopes}
Let $K$ be a positive ($\tau(K)>0$) L-space knot of genus-$g.$ If $S^3_K(\frac{p}{q})\cong -S^3_K(\frac{p}{q'})$ with $p,q>0,$ $q'<0,$ then $\frac{p}{q}<2g-1$ and $q'=-q+\frac{p}{2g-1}.$
\end{lem}
\begin{proof}
By Corollary \ref{cor:LS}, $\tau(K)=g(K)=g,$ and so, we apply Theorem \ref{thm:main} to conclude that $\frac{p}{q}<2g-1$ and $(\V(K)+2g-1)(q+q')=2p.$ Corollary \ref{cor:LS} also implies that $\V(K)=2g-1,$ so we find that $(2g-1)(q+q')=p,$ from which the second conclusion follows.
\end{proof}
We now turn our attention to how the gradings compare between two putative cosmetic surgeries. %Suppose $S^3_K(\frac{p}{q})\cong -S^3(\frac{p}{q'}).$ Then for each $\mathfrak{s}\in Spin^c(S^3_K(\frac{p}{q})),$ $\widehat{HF}(S^3_K(\frac{p}{q}),\mathfrak{s}) = - \widehat{HF}(S^3_K(\frac{p}{q'})),\overline{\mathfrak{s}})$ as graded vector spaces, where $\overline{\mathfrak{s}}$ denotes the natural Spin$^c$-structure on the orientation-reversed manifold associated to $\mathfrak{s}.$ In particular, assuming $p,q>0$ and $q'<0$ as usual, we see from Figure \ref{fig:ref} that the generators of $-\widehat{HF}(S^3_K(\frac{p}{q'})$ in any Spin$^c$-structure may be arranged to be in $\mathcal{Z}$. From Figure \ref{fig:LStransform}, we also see that the bigons involved in computing the relative gradings are the same for the slopes $\frac{p}{q}$ and $\frac{p}{q'},$ with the difference arising from which kind of puncture is taken into account, as well as a minus sign, which is accounted for by the reflection.  Hence, we may simply compare the counts of the two different kinds of punctures occurring within a diagram such as in Figure \ref{fig:start}. 
 To do this, we first define an operation $S_i:\mathbf{Z}^{2n+1}\rightarrow\mathbf{Z}^{2n+1}$ for $1\leq i\leq 2n$ defined by:
\[
S_i\left((m_j)_{j=1}^{2n+1}\right) = \begin{cases}
(m_1,\dots,m_i,m_{i+1} -2, \dots, m_{2n+1}-2) & \text{for }i\text{ even}\\
(m_1,\dots,m_i,m_{i+1} +2, \dots, m_{2n+1}+2) & \text{for }i\text{ odd}
\end{cases}
\]
\begin{lem}\label{lem:shifts}
Let $K$ be a positive L-space knot of genus $g,$ let $p,q>0,$ and put $q'=-q+\frac{p}{2g-1}.$ For each $\mathfrak{s}\in Spin^c(S^3_K(\frac{p}{q})),$ there exists a sequence $i_1,\dots,i_N$ (possibly empty) as well as arrangements $(x_j)_{j=1}^{2n+1}$ and $(y_j)_{j=1}^{2n+1}$ of the generators of $\widehat{HF}(S^3_K(\frac{p}{q}),\mathfrak{s})$ and $\widehat{HF}(S^3_K(\frac{p}{q'}),\psi(\mathfrak{s}))$ such that:
\[
-(M(y_1),\dots,M(y_{2n+1})) \sim S_{i_1}\circ\dots\circ S_{i_N}(M(x_1),\dots,M(x_{2n+1}))
\]
\end{lem}
\begin{proof}
We choose a lift $\widetilde{L^{\mathfrak{s}}}_{p,q} \subset \widetilde{T}_{\bullet\ast}$ corresponding to $\mathfrak{s}.$ We label the generators $x_1,\dots,x_{2n+1}\in \widehat{HF}(S^3(\frac{p}{q}),\mathfrak{s})$ and $y_1,\dots,y_{2n+1}\in \widehat{HF}(S^3(\frac{p}{q'}),\psi(\mathfrak{s}))$ corresponding to the intersection points $\widetilde{\Gamma}\cap \widetilde{L^{\mathfrak{s}}}_{p,q}$ and $r_V(\Phi(\widetilde{\Gamma}\cap \widetilde{L^{\mathfrak{s}}}_{p,q})),$ respectively, ordered from left to right. Let us also number the bigons between successive generators from left to right; we see that $r_V\circ \Phi$ provides an identification between the bigons occurring in the two kinds of intersection diagrams, so we shall consider just one kind of diagram, as in Figure \ref{fig:LSbigon2}. Equation \eqref{eq:grading} gives the relative gradings $(M(x_1),\dots,M(x_{2n+1}))$ as well as $-(M(y_1),\dots,M(y_{2n+1}))$ as follows: let $k_i$ be the number of $\bullet$ punctures covered by the $i$th bigon, and similarly let $k'_i$ be the number of $\ast$ type punctures covered by that bigon. Up to an overall shift, we may take $M(x_1)=-M(y_1).$ Notice that, since $\frac{p}{q}>0,$ $k_i\leq k'_i,$ and whenever $k'_i>k_i,$ we see that there is a contribution of $\pm 2(k'_i-k_i)$ to $M(x_j)+M(y_j)$ for all $j\geq i,$ with the sign depending on the parity of $i.$ In fact, with careful consideration of the signs and parties, we find that that the shift corresponds precisely to $(k'_i-k_i)$ applications of $S_i$ to $(M(x_j))_{j=1}^{2n+1}.$
\end{proof}
\begin{figure}
\centering
\includegraphics[height=.4\textheight]{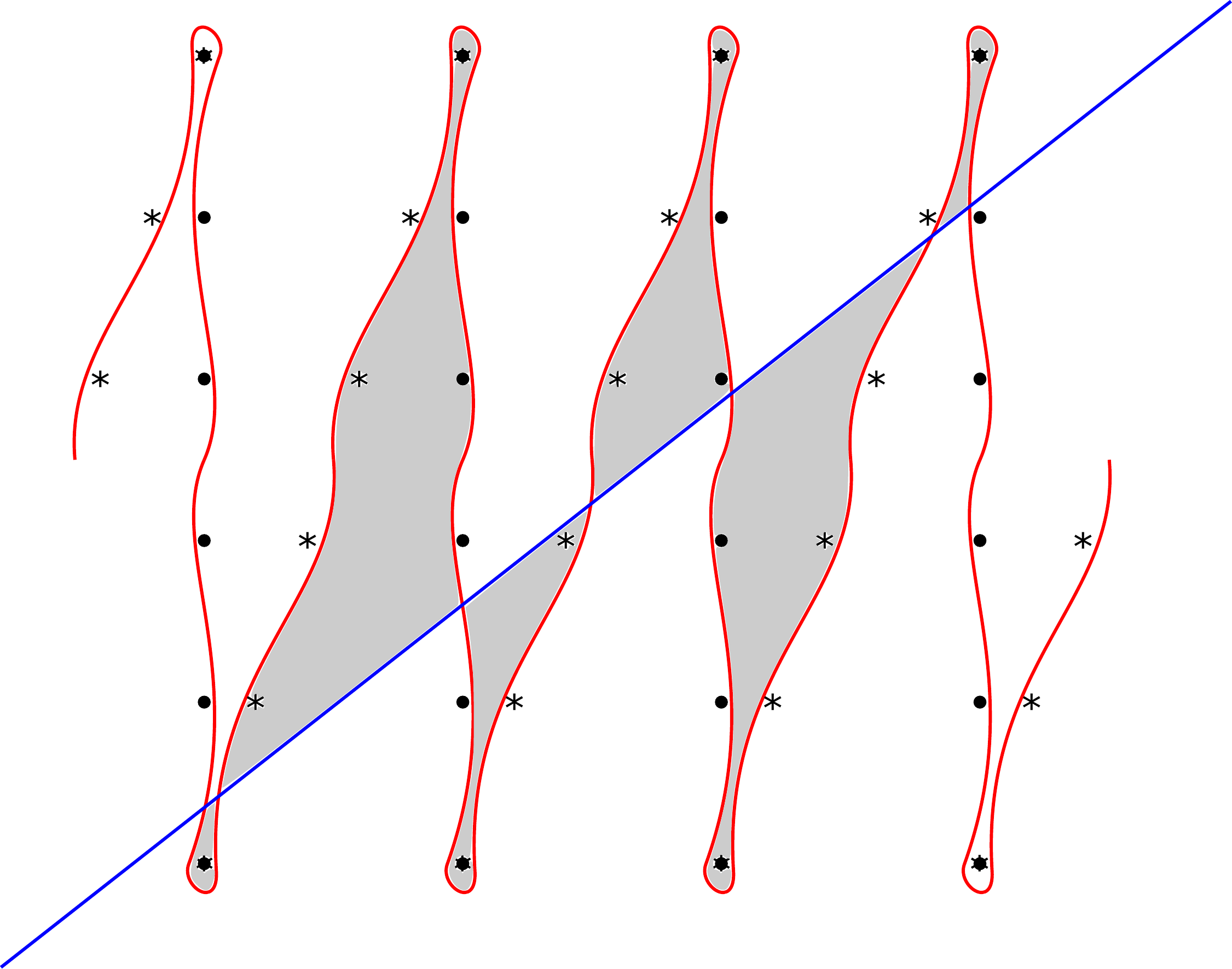}
\caption{An example of relative grading computation for a pair of Spin$^c$-structures related by $\psi.$}
\label{fig:LSbigon2}
\end{figure}

\begin{claim}\label{cla:shift}
Moreover, if $g\geq2,$ there exists an $\mathfrak{s}\in Spin^c(S^3_K(\frac{p}{q}))$ such that the corresponding sequence $i_1,\dots,i_N$ in the above lemma is nonempty.
\end{claim}
\begin{proof}
From the proof of the lemma, it suffices to show that there exists a bigon covering more $\ast$ punctures than $\bullet$ punctures. We make use of the following fact: for some sufficiently small $\epsilon <\frac{1}{q},$ we may take the lifts of $L_{p/q}$ to pass through points of the form $(i,\frac{j}{q}+\frac{1}{2}+\epsilon)$ for $i,j\in \mathbf{Z}$ so that minimal intersection is achieved and none of the lifts intersect any of the punctures. Notice that the $\bullet$ punctures are exactly located at $(i,j+\frac{1}{2})$ for $i,j\in \mathbf{Z},$ whereas the $\ast$ punctures are located at $(i+\frac{j+g}{2g-1},j+\frac{1}{2}).$ By the symmetry of $\widetilde{\Gamma},$ not all $\ast$ punctures will be (locally) to the right/below $\widetilde{\Gamma},$ and moreover, since $g>2,$ there will be at least one $\ast$ puncture locally to the left/above $\widetilde{\Gamma}$ at a height strictly between $-g+\frac{1}{2}$ and $g-\frac{1}{2},$ say at position $(\frac{\ell +g}{2g-1},\ell+\frac{1}{2})$ for some $\ell\in(-g,g-1).$ Since lifts of $L_{p/q}$ are lines of slope $\frac{p}{q},$ they pass through points of the form $(\frac{i}{p},\frac{i+j}{q}+\frac{1}{2}+\epsilon)$ for $i,j\in \mathbf{Z}.$ By Lemma \ref{lem:LSslopes}, we see that $(2g-1)(q+q')=p,$ so that $p$ is a multiple of $(2g-1).$ Hence, there is a lift (which corresponds to some Spin$^c$-structure) passing through the point $(\frac{\ell+g}{2g-1},\ell+\frac{1}{2}+\epsilon).$ This lift will intersect the non-vertical part of $\widetilde{\Gamma}$ near the $\ast$ puncture, and it will also intersect the vertical part on the left and below height $\ell+\frac{1}{2}.$ Thus, a bigon is produced which covers the $\ast$ puncture at that height but not the $\bullet$ puncture; see Figure \ref{fig:LS_strict} for an illustration. Of course, for each $\bullet$ puncture at a lower height that is covered, the $\ast$ puncture directly to the right will also be covered, as we saw in the proof of the previous lemma, and so this bigon covers strictly more $\ast$ punctures, as desired. 
\end{proof}
\begin{figure}
\centering
\includegraphics[height=.4\textheight]{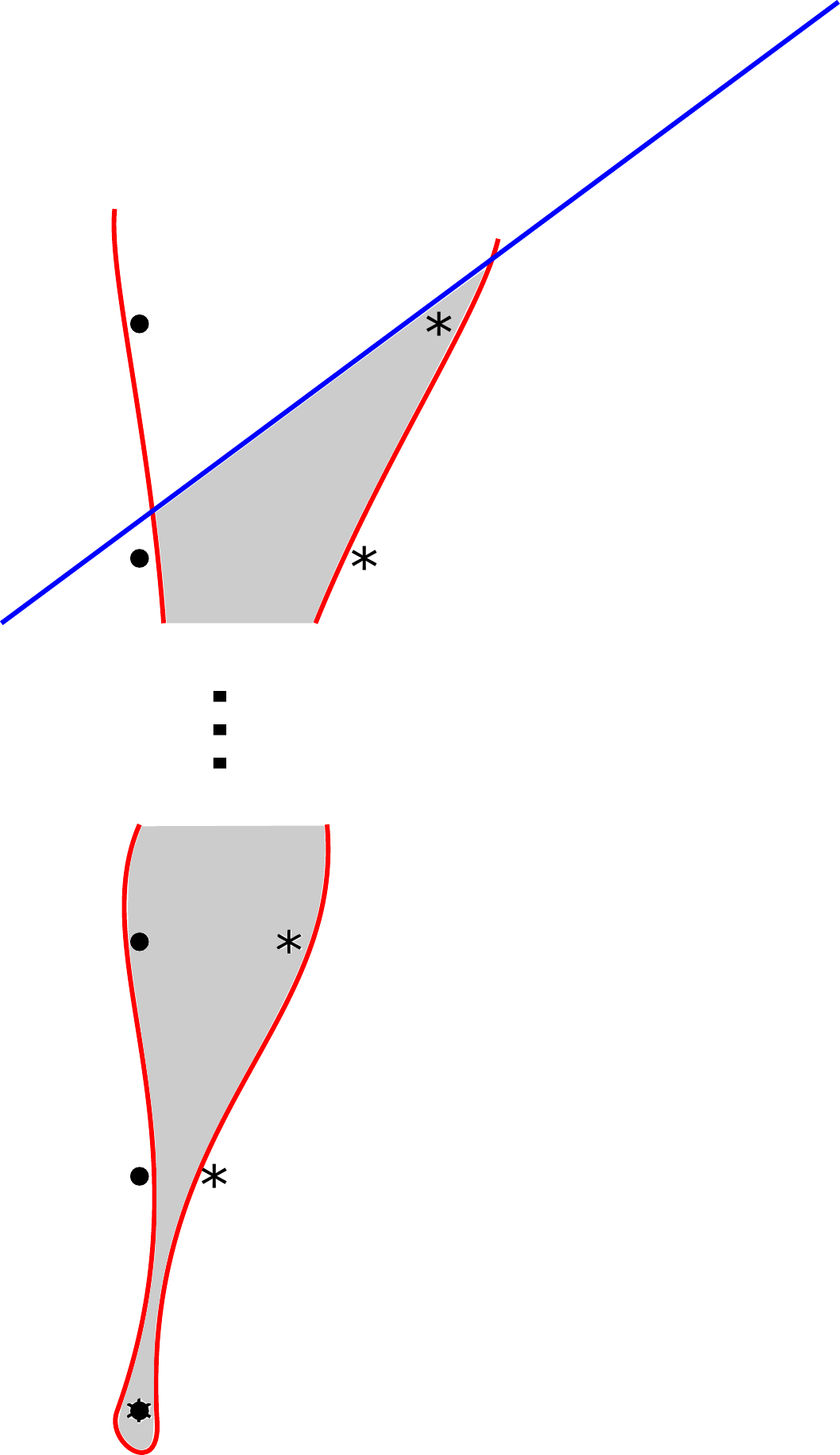}
\caption{An example of a bigon covering strictly more $\ast$ punctures than $\bullet$ punctures.}
\label{fig:LS_strict}
\end{figure}
We now wish to show that applying $S_i$ to some of the relatively graded sets coming from $\widehat{HF}(S^3_K(\frac{p}{q}))$ results in inequivalent relative gradings. First, we define $\Delta_{avg}:\mathbf{Z}^{2n+1}\rightarrow \mathbf{Q}$ as follows:
\[
\Delta_{avg}\left((m_i)_{i=1}^{2n+1}\right) := \frac{1}{n}(m_2+m_4+\dots+m_{2n})-\frac{1}{n+1}(m_1+m_3+\dots+m_{2n+1}).
\]
We claim that this descends to a well-defined map $\Delta_{avg}:\mathcal{Z}_n\rightarrow \mathbf{Q}.$ This is because any representative of an equivalence class in $\mathcal{Z}$ must be a sequence which alternates in parity; that is, $m_i \equiv m_j \mod 2 \Leftrightarrow i \equiv j \mod 2,$ and so we interpret $\Delta_{avg}$ as the difference between the averages of the two subsets partitioned by the mod 2 grading. In addition, since it is a difference of averages, it is also invariant under an overall shift, making it well-defined on relatively graded sets.
\begin{lem}\label{lem:avg_shift}
For any $(m_j)_{j=1}^{2n+1}\in \mathbf{Z}^{2n+1}$ with $n>0$ and $1\leq i\leq 2n,$
\[
\Delta_{avg}\left((m_j)_{j=1}^{2n+1}\right) < \Delta_{avg}\left(S_i\left((m_j)_{j=1}^{2n+1}\right)\right).
\]
\end{lem}
\begin{proof}
If $i$ is even, since $\Delta_{avg}$ is invariant under an overall shift, we see that:
\begin{align*}
\Delta_{avg}\left(S_i\left((m_j)_{j=1}^{2n+1}\right)\right) &= \Delta_{avg}\left(m_1,\dots,m_i,m_{i+1}-2,\dots,m_{2n+1}-2 \right)\\
&= \Delta_{avg}\left(m_1+2,\dots,m_i+2,m_{i+1},\dots,m_{2n+1} \right)\\
&=\frac{1}{n}(i+m_2+m_4+\dots+m_{2n})-\frac{1}{n+1}(i+m_1+m_3+\dots+m_{2m+1})\\
&=\frac{i}{n}-\frac{i}{n+1}+\Delta_{avg}\left((m_j)_{j=1}^{2n+1}\right)\\
&>\Delta_{avg}\left((m_j)_{j=1}^{2n+1}\right)
\end{align*}
If $i$ is odd, we put $k=2n+1-i$ and observing that $1\leq k \leq 2n,$ we have
\begin{align*}
\Delta_{avg}\left(S_i\left((m_j)_{j=1}^{2n+1}\right)\right) &= \Delta_{avg}\left(m_1,\dots,m_i,m_{i+1}+2,\dots,m_{2n+1}+2 \right)\\
&=\frac{1}{n}(k+m_2+m_4+\dots+m_{2n})-\frac{1}{n+1}(k+m_1+m_3+\dots+m_{2m+1})\\
&=\frac{k}{n}-\frac{k}{n+1}+\Delta_{avg}\left((m_j)_{j=1}^{2n+1}\right)\\
&>\Delta_{avg}\left((m_j)_{j=1}^{2n+1}\right)
\end{align*}
\end{proof}

\begin{proof}[Proof of Theorem \ref{thm:lspace}]
Since $K$ is nontrivial, by Corollary \ref{cor:LS}, $\tau(K)\neq 0$ (as usual, we may assume it to be positive), so that $K$ admits no purely cosmetic surgeries by \cite{NW}. So we only need to consider chirally cosmetic surgeries. If $g(K)=1,$ then $K$ is a trefoil, whose chirally cosmetic surgeries are known to be along slopes of the same sign by \cite{IIS}. Hence, we consider $g(K)\geq 2.$  Suppose $K$ admits chirally cosmetic surgeries along slopes $r>0$ and $r'<0.$ For each $\mathfrak{s}\in Spin^c(S^3_K(r)),$ denote the relatively graded set of generators of $\widehat{HF}(S^3_K(r),\mathfrak{s})$ by $\mathcal{G}^r_{\mathfrak{s}};$ similarly, denote the relatively graded set of generators of $\widehat{HF}(S^3_K(r'),\mathfrak{t})$ by $\mathcal{G}^{r'}_{\mathfrak{t}}$ for each $\mathfrak{t}\in Spin^c(S^3_K(r')).$ By Lemmas \ref{lem:LSslopes}, \ref{lem:shifts}, and \ref{lem:avg_shift}, for each $\mathfrak{s}\in Spin^c(S^3_K(r)),$ $\Delta_{avg}\left(\mathcal{G}^r_{\mathfrak{s}}\right) \leq \Delta_{avg}\left(-\mathcal{G}^{r'}_{\psi(\mathfrak{s})}\right).$ Moreover, by Claim \ref{cla:shift}, there exists an $\mathfrak{s}\in Spin^c(S^3_K(r))$ such that $\Delta_{avg}\left(\mathcal{G}^r_{\mathfrak{s}}\right) < \Delta_{avg}\left(-\mathcal{G}^{r'}_{\psi(\mathfrak{s})}\right).$ However, since $S^3_K(r)\cong-S^3_K(r'),$ we must have that $\mathcal{G}^r_{\mathfrak{s}} = -\mathcal{G}^{r'}_{\overline{\mathfrak{s}}}$ for each $\mathfrak{s}\in Spin^c(S^3_K(r)).$ This is in contradiction with the previous inequalities.
\end{proof}

\section*{Acknowledgments}
The author would like to thank Zolt\'{a}n Szab\'{o} and Peter Ozsv\'{a}th for encouraging work on this project and for comments on an earlier draft. The author also thanks Jonathan Hanselman for many helpful discussions.  This work was supported by NSF grants DMS-1502424 and DMS-1904628.

\printbibliography

\end{document}